\documentclass[a4paper,10pt]{article}
\usepackage{hyperref}
\usepackage{amsfonts}
\usepackage{amssymb}
\usepackage{amsthm}
\usepackage{amsmath}
\usepackage[shortcuts]{extdash}
\usepackage{color}
\usepackage[matrix,arrow,curve]{xy}
\allowdisplaybreaks

\oddsidemargin=-\evensidemargin
\newtheorem{theorem}[equation]{Theorem}
\newtheorem{lemma}[equation]{Lemma}
\newtheorem{remark}[equation]{Remark}
\newtheorem{corollary}[equation]{Corollary}
\newtheorem{proposition}[equation]{Proposition}
\numberwithin{equation}{section}

\theoremstyle{definition}
\newtheorem{definition}[equation]{Definition}
\newtheorem{example}[equation]{Example}

\theoremstyle{remark}
\newtheorem{sideremark}[equation]{Side-remark}

\newcommand{\ZZ}{\mathbb Z}

\newcommand{\X}{\mathfrak X}

\newcommand{\Div}{\mathop{\mathrm{Div}}}

\newcommand{\Hom}{\mathop{\mathrm{Hom}}\nolimits}
\newcommand{\Ext}{\mathop{\mathrm{Ext}}\nolimits}
\newcommand{\Isom}{\mathop{\mathrm{Isom}}\nolimits}
\newcommand{\Mor}{\mathop{\mathrm{Mor}}\nolimits}

\newcommand{\inv}{\mathrm{inv}}
\newcommand{\id}{\mathrm{id}}
\newcommand{\pt}{\mathrm{pt}}
\newcommand{\pr}{\mathrm{pr}}
\newcommand{\Spec}{\mathop{\mathrm{Spec}}}
\newcommand{\Span}{\mathop{\mathrm{Span}}\nolimits}

\newcommand{\CH}{\mathop{\mathrm{CH}}}
\newcommand{\Pic}{\mathop{\mathrm{Pic}}}

\newcommand{\GalDec}{\mathop{\mathrm{Dec}}\nolimits}

\newcommand{\triv}{\mathop{\mathbf{t}}\nolimits}
\newcommand{\trivinv}{\mathop{\mathbf{d}}\nolimits}
\newcommand{\trivinvinv}{\mathop{\mathbf{i}}\nolimits}
\newcommand{\Gal}{\mathop{\mathrm{Gal}}}

\newcommand{\Aut}{\mathop{\mathrm{Aut}}}

\newcommand{\cd}{\mathop{\mathrm{cd}}\nolimits}
\newcommand{\cdg}{\mathop{\mathfrak{cd}}\nolimits}

\newcommand{\Br}{\mathop{\mathrm{Br}}\nolimits}

\newcommand{\trdeg}{\mathop{\mathrm{trdeg}}\nolimits}

\newcommand{\Sch}{\mathcal Sch}
\newcommand{\StSch}{\mathcal St\mathcal Sch}

\newcommand{\kdel}[1]{}

\newcommand*{\uline}[1]{{\lefteqn{#1}\underline{\hphantom{#1}}}}

\renewcommand{\appendixname}{Appendix}

\makeatletter
\def\@seccntformat#1{\csname the#1\endcsname. }
\makeatother

\begin{document}

\title{An estimate of canonical dimension of groups based on Schubert calculus}
\author{Rostislav Devyatov\thanks{Laboratory of Algebraic Geometry and its Applications,
Department of Mathematics,
National Research University Higher School of Economics,
6 Usacheva str.,
Moscow 119048,
Russian Federation.\newline\textit{Email address:} \texttt{deviatov@mccme.ru} .}}
\maketitle

\begin{abstract}
We sketch the proof of a connection between the canonical (0-)dimension of semisimple split simply connected groups and cohomology of their full flag varieties.
Using this connection, we get a new estimate of the canonical (0-)dimension of simply connected split exceptional groups of type $E$ understood as a group.
\end{abstract}

\section{Introduction}
To define the canonical (0-)dimension of an algebraic group understood as a group, 
we first need to define the canonical (0-)dimension of a scheme understood as a scheme (which is a different definition).
Roughly speaking, the canonical (0-)dimension of a scheme is a number indicating how hard it is to get a rational point in the scheme.
The canonical (0-)dimension of an algebraic group shows how hard it is to get rational points in torsors related to the group.

To be more precise, let us fix some conventions and give some definitions.
All schemes in the present text are understood in the usual sense of algebraic geometry, are of finite type over a fixed base field, and are separated.
The words ``scheme'' and ``variety'' are used interchangeably, in the same sense.
The base field is arbitrary.

Speaking of canonical dimension of schemes, there are two closely related notions in the literature:
the \emph{canonical 0-dimension of a scheme} defined in \cite{merkurjevupd3}
and the \emph{canonical dimension of a scheme} defined in \cite{karpenkomerkurjevadvances}.
These two definitions are not known to be always equivalent, but they are equivalent 
for two particular classes of schemes:
for smooth complete schemes and for torsors of split reductive groups
(see \cite[Theorem 1.16, Remark 1.17, and Example 1.18]{merkurjevcontmath}).
The definition from \cite{merkurjevupd3} looks more motivated, so we are going to use it.
\begin{definition}[{\cite[Section 4a, first paragraph of Section 4b, and the last paragraph of Section 2a]{merkurjevupd3}}]\label{candimdefaultdef}
Given a scheme $X$ over a field $K$, the canonical 0-dimension of $X$ understood as a scheme 
(notation: $\cd(X)$)
is:
\begin{equation*}
\cd (X) = \max_{\substack{L = {}\text{a field containing $K$}\\X_L\text{ has a rational point}}} \,\,\,
\min_{\substack{L_0 = {}\text{a subfield of $L$}, \,\,K \subseteq L_0 \\ X_{L_0} \text{ still has a rational point}}}
\trdeg_K{L_0}.
\end{equation*}
This maximum is known to be always finite, moreover, $\cd (X) \le \dim X$, 
see \cite[\S 4b]{merkurjevupd3} or \cite[\S 1]{merkurjevcontmath}.
\end{definition}
A bit less formally, canonical dimension can be explained as follows. Suppose we have expanded the base field
$K$ to $L$, and got a rational point in $X_L$. How large can $L$ be, compared to $K$? In general, it can be very large, 
this is unbounded. 
A related question
with a finite answer is: how many algebraically independent generators do we have 
to keep, at worst (for the worst $L$), to still have a rational point after the scalar extension
(\emph{not necessarily the same} rational point that we found after expanding scalars to $L$)? This number of generators 
is the canonical dimension of $X$. For more properties of canonical dimension, see \cite{merkurjevupd3} in the case of general $X$
and \cite{karpenkoimc} in the case of smooth projective $X$.

We have underlined above that we want to get a rational point over a field between $K$ and $L$, but not necessarily the same rational point.
If we demanded to get the same rational point, we would get the definition of the \emph{essential 0-dimension} of a scheme, 
which is known to coincide with the (standard in algebraic geometry) dimension, see \cite[Proposition 1.2]{merkurjevcontmath}.
This can be viewed as a motivation for the word ``dimension''. 
(But essential dimension is not only defined for schemes, and in broader generality it becomes a much more nontrivial notion.)


Another motivation for canonical (0-)dimension comes from \emph{incompressible varieties}.
This motivation is only valid for 
the canonical (0-)dimension of smooth complete schemes.
A smooth 
complete
variety $Y$ is called \emph{incompressible} if any rational map $Y \dashrightarrow Y$ is dominant.
If $X$ is a smooth complete variety, then by \cite[Theorem 1.7]{merkurjevcontmath}, $\cd (X)$ is the smallest dimension 
of the image of a rational map $X \dashrightarrow X$. 
In other words, $\cd (X)$ is the smallest dimension of an incompressible subvariety 
$Y \subseteq X$ such that there exists a rational map $X \dashrightarrow Y$.

The second object we need to define before we can define the canonical 0-dimension of a group is a torsor of a group.
All algebraic groups in this text are affine.
All reductive, semisimple, and simple groups in this text are smooth and connected.
Torsors of algebraic groups (over a point) are, informally speaking, homogeneous spaces that are ``as large as the group itself''. 
This notion is mostly interesting over non-algebraically closed fields.
\begin{definition}[{\cite[Section 3a]{merkurjevupd3}}]
Given an algebraic group $G$, a \emph{$G$-torsor over a point} (or simply a \emph{$G$-torsor})
is a scheme $E$ with an action $\varphi \colon G \times E \to E$
such that\footnote{In this paper, we suppose that all actions of algebraic groups, including the actions on torsors, satisfy 
the ``left'' (``standard'') axiom of action. Formally: if we denote by $\varphi_g := \varphi|_{\{g\} \times E} \colon E \to E$ the action of 
a single element (a rational point) $g \in G$, then $\varphi_{g_1 g_2} = \varphi_{g_1} \circ \varphi_{g_2}$ for any $g_1, g_2 \in G$ 
(and not $\varphi_{g_1 g_2} = \varphi_{g_2} \circ \varphi_{g_1}$). 
Informally (or in hom-functorial/Yoneda notation): $(g_1 g_2) \cdot e = g_1 \cdot (g_2 \cdot e)$, not $e^{g_1 g_2} = (e^{g_1})^{g_2}$.
This is not important here in the Introduction, but will be important when we work with this action explicitly later.}
$(\varphi, \mathrm{pr}_2) \colon G \times E \to E \times E$, where $\mathrm{pr}_2$ is the projection to 
the second factor, is an isomorphism.
\end{definition}
It is known that all torsors of affine algebraic groups over a point are affine.
However, torsors quite often don't have any rational points.

Finally, the canonical 0-dimension of an algebraic group understood as a group measures 
how hard it is to get rational points in torsors, informally speaking, related to the group.
Precisely:
\begin{definition}[{\cite[Section 4g]{merkurjevupd3}}]\label{candimgroup}
Given an algebraic group $G$ over a field $F$, 
the canonical 0-dimension of $G$ understood as a group (notation: $\cdg (G)$)
is
\begin{equation*}
\cdg(G) = \max_{K = \text{a field containing $F$}} \,\,\, \max_{E = \text{a $G_K$-torsor}} \cd (E).
\end{equation*}
Again, the (usual) dimension of a torsor is known to be always equal to the dimension of the group, so here, the maximum is also finite, 
and also $\cdg(G) \le \dim G$.
\end{definition}

The definition of canonical dimension of an algebraic group understood as a group in \cite[Introduction]{karpenkomerkurjevadvances}
repeats this definition almost exactly, with the only difference being that 
instead of $\cd (E)$ it uses the definition of canonical dimension of $E$ understood as a scheme 
from the paper \cite{karpenkomerkurjevadvances} itself. But as we already mentioned above, 
it is known that these two notions are 
equivalent for torsors of split reductive groups.
So, Definition \ref{candimgroup} is also equivalent to the definition of canonical dimension of a group from 
\cite[Introduction]{karpenkomerkurjevadvances} for split reductive groups. 
All groups whose canonical dimension we are going to estimate in this text are split reductive (and even simply connected semisimple),
so these results also estimate the canonical dimension in the sense of \cite[Introduction]{karpenkomerkurjevadvances}.

Speaking of the motivation for the canonical dimension of a variety coming from incompressible varieties, 
its relation to Definition \ref{candimgroup} is not clear right away: the varieties whose canonical dimension we 
measure in Definition \ref{candimgroup} are $G_K$-torsors, and 
they are affine, 
not complete.
But a connection with incompressible varieties still exists: if $G$ is simply connected and split semisimple (in fact, split reductive is enough), 
given a $G_K$-torsor $E$, one can define 
\emph{the quotient of $E$ modulo a Borel subgroup} (see Definition \ref{defemodb}). 
This quotient is smooth projective, and it is known that its canonical dimension equals $\cd(E)$ (see Corollary \ref{cdeebequal}).
So, we can say informally and approximately that 
the canonical dimension of a group measures how large, depending on a torsor, 
the smallest incompressible subvariety in the quotient of the torsor modulo the Borel subgroup can be. 
And precisely, we don't look at all incompressible subvarieties $Y$ of a given quotient $X$, 
but only at such $Y$s that there exists a rational map $X \dashrightarrow Y$.

Now, with all of this motivation, before we can formulate the main goal of this text precisely,
we need to introduce some more notation and terminology.
Given a split semisimple algebraic group $G$ and a Borel subgroup $B$, the corresponding Weyl group $W$, and 
the element $w_0 \in W$ of maximal length, 
for each $w \in W$ we denote the Schubert variety $\overline{B w_0 w^{-1} B / B} \subseteq G/B$ by $Z^w$.
This $Z^w$ is a Schubert divisor if and only if $w$ is a simple reflection, and we denote all Schubert divisors 
by $D_1, \ldots, D_r$, where $1,\ldots, r$ are the indices of
the simple roots of $G$.

It is known that the classes $[Z^w] \in \CH^* (G/B)$ for all $w \in W$ form a free set of generators of $\CH^*(G/B)$ as 
of an abelian group.
We say that a product of classes of Schubert divisors 
$[D_1]^{n_1}\ldots [D_r]^{n_r}$
is \emph{multiplicity-free} if there exists $w \in W$ 
such that the coefficient at $[Z^w]$ 
in the decomposition of $[D_1]^{n_1}\ldots [D_r]^{n_r}$
into a linear combination of Schubert classes equals 1.

Now we can formulate the goal of this text precisely. Our goal is to sketch the proof of the following theorem.
\begin{theorem}\label{maintheorem}
Let $G$ be a split semisimple simply connected algebraic group over an arbitrary field, let $B$ be a Borel subgroup, 
let $r$ be the rank of $G$, and let $D_1, \ldots, D_r \subset G/B$ be 
the Schubert divisors corresponding to the $r$ simple roots of $G$.
If $[D_1]^{n_1}\ldots [D_r]^{n_r}$ is a multiplicity-free product of Schubert divisors, 
then $\cdg (G) \le \dim (G/B) - n_1 -\ldots - n_r$.
\end{theorem}
%
As a corollary of this theorem and \cite[Theorem 8.6]{multiplicityfreepreprint}, we will immediately get the following.
\begin{corollary}\label{corollarye}
Let $G$ be a split simple simply connected algebraic group of type $E_r$.
Then $\cdg (G) \le 17$, $37$, or $86$ for $r = 6$, $7$, or $8$, respectively.\qed
\end{corollary}

The most difficult part of estimating the canonical dimension of simply connected split groups of type 
$E_r$ (and in obtaining Corollary \ref{corollarye}) 
was actually to understand which products of Schubert divisors are multiplicity-free 
(and this was understood in \cite{multiplicityfreepreprint} by the author). 
The description of multiplicity-free products of Schubert divisors in \cite{multiplicityfreepreprint}
is explicit enough to find the maximal degree of such a multiplicity-free product precisely.
However, for the canonical dimension we still get only an estimate from above, 
because Theorem \ref{maintheorem} can only produce upper estimates anyway.
In other words, in view of Theorem \ref{maintheorem}, a lower bound 
of the maximal degree of a multiplicity-free product of Schubert divisors implies an upper bound of the canonical dimension of the group.
And an upper bound 
of the maximal degree of a multiplicity-free product of Schubert divisors puts a lower limit on the upper bounds of the canonical dimension 
that can be obtained with this technique.

The part of the argument establishing relation between Schubert calculus and canonical dimension 
(in other words, the proof of Theorem \ref{maintheorem} itself)
was known to the experts 
in the area
(or at least they believed that the argument is doable this way).
However, we were unable to find an exposition suitable for more general mathematical audience.
The present paper contains such an exposition.
In this text, we are going to follow the ideas of several proofs from \cite{karpenkoduke},
where canonical dimension was related to cohomology of flag varieties of orthogonal groups 
(more precisely, orthogonal Grassmannians, not full flag varieties).

Speaking of 
the canonical dimension of simply connected split simple groups of other types, 
in types $A_r$ and $C_r$ the canonical dimension is known to be zero. 
For types $B_r$ and $D_r$, the canonical dimension was estimated (and computed exactly if $r$ is a power of 2) 
by N.~Karpenko in \cite{karpenkoduke}.
In type $D_r$, even though the maximal degree of a multiplicity-free product of Schubert divisors 
is also found precisely in \cite{multiplicityfreepreprint},
the resulting estimate of the canonical dimension from Theorem \ref{maintheorem} 
turns out to be the same as Karpenko's estimate ${} \le (r-1)(r-2)/2$.
Karpenko's estimate for type $B_r$ is $r(r-1)/2$.
For type $G_2$, the canonical dimension (of a split simply connected group) is known and equals 3, see \cite[Example 10.7]{berhuyreichstein}.

After (the first version of) this text was posted on arXiv, the examples of multiplicity-free monomials 
found in \cite{multiplicityfreepreprint} for groups of type $A$, $D$, and $E$, were generalized to 
groups of other types in 
\cite{zaynullinbcfg}.
In particular, in \cite{zaynullinbcfg}, a monomial
of degree 11 for type $F_4$ was found, and therefore, an estimate ${} \le 13$
for the canonical dimension of a split simply connected group $G$ of type $F_4$ was obtained. However, 
the question of the maximal possible degree of a multiplicity-free monomial in type $F_4$ remains open. 
Therefore, it might be possible to get a better estimate for $\cdg(G)$ using the same technique of multiplicity-free monomials and
Theorem \ref{maintheorem} if one finds a multiplicity-free monomial of a higher degree (higher than 11).
The same open question exists in type $B_r$.

The main results of the present text were preliminarily announced in a short note \cite{multiplicityfreecandimshortnote} by the author.

\subsection*{Acknowledgments}
I thank Kirill Zainoulline for bringing my attention to the problem.
I thank Nikita Karpenko for useful discussions and explanations about theory of torsors and canonical dimension 
and about his paper \cite{karpenkoduke}.
I also thank Vladimir Chernousov and Alexander Merkurjev
for useful discussions about torsors and theory of Galois descent.
I also thank the following people for discussions about intersection theory and algebraic geometry in general:
Stephan Gille, 
Nikita Semenov, 
Alexander Vishik, and Bogdan Zavyalov.
I thank Peter Scholze for useful discussions of the categorical approach to 
algebraic geometry over non-algebraically closed fields.
I thank Ivan Panin for suggesting me to think about Galois descent for line bundles, which helped me to simplify 
the proof of Proposition \ref{propositionpicards}.
I thank Marat Rovinskiy for useful discussions about \'etale sheaves.

\subsection*{Funding}
This research was partly supported by the Pacific Institute for the Mathematical Sciences fellowship.
The author also thanks Max Planck Institute for Mathematics in Bonn for its financial support and hospitality.
This work was partially completed at the Department of Mathematical Sciences, KAIST and 
was partially supported by Samsung Science and Technology Foundation under Project Number SSTF-BA1901-02.
Support from the Basic Research Program of HSE University is gratefully acknowledged (HSE-BR-2025-061),
and the author especially thanks Laboratory of Algebraic Geometry and its Applications of HSE University for hospitality.

\section{Preparation 1: Recall of basic Galois descent theory}
%
To study torsors (and schemes over a non-algebraically-closed field in general), 
we will make a lot of use of Galois descent theory. 
We will need two versions of this theory: for vector spaces and for schemes.

The version for vector spaces is quite simple. Suppose we have a finite Galois extension of fields $K \subseteq L$ with Galois group $\Gamma$.
\begin{definition}
Let $V$ and $W$ be two $L$-vector spaces, and let $\sigma \in \Gamma$. A map (of sets) $f \colon V \to W$ is called 
\emph{$\sigma$-semilinear} if $f (a_1 v_1 + a_2 v_2) = \sigma (a_1) f(v_1) + \sigma(a_2) f(v_2)$ for all $a_1, a_2 \in L$ and $v_1, v_2 \in V$.
\end{definition}
\begin{definition}
Let $V$ be an $L$-vector space. A \emph{semirepresentation} of $\Gamma$ on $V$ 
is an action $\psi \colon \Gamma \times V \to V$ on $V$ \textbf{understood as a set} such that 
for each $\sigma \in \Gamma$, the map $\psi|_{\{\sigma\} \times V} \colon V \to V$ is $\sigma$-semilinear.
\end{definition}
\begin{example}\label{stdrepresuv}
Let $U$ be a $K$-vector space. Then we can define a $\Gamma$-semirepresentation on $V = L \otimes_K U$ by the formula 
$\sigma (a \otimes u) = \sigma (a) \otimes u$ for all $a \in L$ and $u \in U$: 
the formula defines a $K$-bilinear map, so it can be extended to the whole $L \otimes_K U$.

This semirepresentation will be called the \emph{standard representation} of $\Gamma$ on $L \otimes_K U$.
\end{example}
Given a semirepresentation of $\Gamma$ on an $L$-vector space $V$, we can define the \emph{dual} semirepresentation
of $\Gamma$ on $V^*$ by the formula $(\sigma f)(v) = \sigma (f (\sigma^{-1}(v)))$ for all $\sigma \in \Gamma$, $f \in V^*$, and $v \in V$. 
A direct computation shows that this action indeed produces elements of $V^*$ out of elements of $V^*$,
and one more direct computation shows that this is a semirepresentation.
We can further induce a semirepresentation of $\Gamma$ on the symmetric algebra $S^\bullet (V^*)$ by saying that 
$\sigma (fg) = (\sigma f)(\sigma g)$.

We will need the following well-known fact about semirepresentations, which is sometimes called Hilbert's Theorem 90.
\begin{theorem}\label{hilbertninety}
Let $K$, $L$, $\Gamma$ be as above.
Suppose we have a semirepresentation of $\Gamma$ on an $L$-vector space $V$. 
Then $V^\Gamma$ is a $K$-vector space, and the (obvious) map $L \otimes_K V^\Gamma \to V$, $a \otimes v \mapsto av$, 
is an isomorphism. \qed
\end{theorem}

Now let us recall the basic notions and facts of Galois descent theory for schemes.
We will need three categories.
The first category, $\Sch_K$, is the category of (separated and of finite type, as everywhere in the text) schemes over a field $K$.

To define the second category, suppose we have a finite extension of fields $K \subseteq L$. 
First, we need to recall the definition of the functor of restriction of scalars from $\Sch_L$ to $\Sch_K$.
If $X$ is an object of $\Sch_L$, we say that 
$X$ \emph{with scalars restricted from $L$ to $K$} is the scheme that has the same topological space as $X$, 
the same ring of regular functions on each open subset \emph{as an abstract ring}, but for the algebra structure, 
we view this ring as a $K$-algebra rather than an $L$-algebra (the multiplication by elements of $K$ is given by the embedding $K \subseteq L$).
We denote this scheme by $X|_K$.
And if $f \in \Mor_{\Sch_L} (X, Y)$, then one can check directly that the same map of topological spaces as in $f$, 
together with the same map of abstract rings for each open subset of $Y$ ($=$ each open subset of $Y|_K$) as in $f$,
satisfies the definition of a morphism of $K$-schemes from $X|_K$ to $Y|_K$.

Now we can say that the second category, which 
we will call \emph{the category of $K,L$-schemes} (notation: $\Sch_{K,L}$),
has schemes over $L$ as objects, and the set $\Mor_{\Sch_{K,L}} (X, Y)$, where $X$ and $Y$ are $L$-schemes, 
is the set of morphisms of $K$-schemes from $X|_K$ to $Y|_K$.

The third category will be introduced a bit later.

\begin{example}\label{sigmastarexample}
Let $K \subseteq L$ be a finite Galois extension of fields, and let $\sigma \in \Gal (L/K)$.
Let $X=\Spec L$. Then $K[X|_K]=L$, and $\sigma^{-1}\colon L \to L$ is an automorphism of this $K$-algebra.
It defines 
the dual automorphism of the $K$-scheme $X|_K$, which we denote by 
$\sigma_* \in \Mor_{\mathcal Sch_{K,L}}(\Spec L, \Spec L)$.

We keep the notation $\sigma_*$ until the end of the text.
\end{example}
\begin{definition}\label{defsemilinearmorphism}
Let $K \subseteq L$ be a finite Galois extension of fields with Galois group $\Gamma$, and let $\sigma \in \Gamma$. 
A morphism $f \colon X \to Y$ in $\mathcal Sch_{K,L}$ is called
\emph{$\sigma$-semilinear} if the following diagram (in $\mathcal Sch_{K,L}$) is commutative:
\begin{equation*}
\xymatrix{
X \ar[r]^{f} \ar[d] & Y \ar[d] \\
\Spec L \ar[r]^{\sigma_*} & \Spec L
}
\end{equation*}
The vertical arrows are the restrictions of scalars of the structure morphisms.
%
\end{definition}
Clearly, under the conditions of this definition, if $f$ (resp.\ $g$) is a $\sigma$ (resp.\ $\tau$)-semilinear morphism, 
then $g \circ f$ is a $(\tau\sigma)$-semilinear morphism.
It is also clear that $1_\Gamma$-semilinear morphisms are exactly the restrictions of scalars of the morphisms in $\Sch_L$.
\begin{definition}[{see \cite[\S V.4.20]{serreclass}}]\label{serretwist}
Let $L/K$ be a finite Galois extension of fields, and let $\sigma \in \Gal (L/K)$.
Let $X$ be an $L$-scheme. Similarly to $X|_K$, let us define the \emph{$L$-scheme obtained from $X$ by means of $\sigma$} 
(notation: $X^\sigma$) as follows. We take the same topological space as for $X$ and the same ring of 
regular functions on each open subset as an abstract ring. But for the $L$-algebra (i.e. $L$-vector space) structure, 
if $U^\sigma \subseteq X^\sigma$ is an open subset identified with $U \subseteq X$, whenever we want to multiply a regular function 
$f \in L[U^\sigma]$ by an element $a \in L$, we say that the result is $\sigma(a)f$ in the sense of the $L$-algebra structure on $L[U]$.

Then the identity map $X^\sigma \to X$ (more precisely, the identity map on the topological space and on all rings of regular 
functions viewed as abstract rings) becomes a $\sigma$-semilinear morphism in $\mathcal Sch_{K,L}$. 
We denote this morphism by $\sigma_{X^\sigma,X} \colon X^\sigma \to X$.
\end{definition}
\begin{definition}\label{defsigmainverse}
Let $K \subseteq L$ be a finite Galois extension of fields with Galois group $\Gamma$.
We will say that we have a 
\emph{Galois-semiaction} of $\Gamma$ on 
an $L$-scheme $X$
(or that $\Gamma$ \emph{Galois-semiacts} on $X$)
if we have an action 
$\psi \colon \Gamma \times X|_K \to X|_K$ 
(here $\Gamma$ is understood as an algebraic group over $K$)
such that for each $\sigma \in \Gamma$, the automorphism 
$\psi|_{\{\sigma\} \times X|_K}$ 
of $X|_K$,
understood as an automorphism of $X$ in $\mathcal Sch_{K,L}$,
is $\sigma$-semilinear.
For simplicity of notation, given a Galois-semiaction $\psi$ and an element $\sigma \in \Gamma$, 
we will 
write $\psi_{\sigma}$ instead of $\psi|_{\{\sigma\} \times X|_K}$.

We say that a finite affine open covering of $X$ is \emph{$\Gamma$-stable} if $\Gamma$ preserves (normalizes) each of these open sets.
\end{definition}

\begin{example}\label{actinducedbyrep}
Let $V$ be a $L$-vector space equipped with a semirepresentation of $\Gamma$. 
Then, informally speaking one can ``extend this semirepresentation to a semiaction of $\Gamma$ on $V$ understood as a scheme''.

More formally, consider the dual semirepresentation of $\Gamma$ on $V^*$ and the induced semirepresentation on $S^\bullet (V^*)$.
Then for each $\sigma \in \Gamma$, 
the action of $\sigma^{-1}$ on $S^\bullet (V^*)$ is a $\sigma^{-1}$-semilinear map of vector spaces 
$S^\bullet (V^*) \to S^\bullet (V^*)$, 
and a direct check shows that the dual morphism of schemes $\Spec (S^\bullet (V^*))|_K \to \Spec (S^\bullet (V^*))|_K$
is a $\sigma$-semilinear morphism in $\Sch_{K,L}$. 
Another direct check shows that these semilinear morphisms together (for all $\sigma \in \Gamma$) form 
a Galois-semiaction on $\Spec (S^\bullet (V^*))$ (which is a formal way of ``viewing $V$ as a scheme'')
and that on the rational points, this Galois-semiaction coincides with the original semirepresentation of $\Gamma$ on $V$ understood as a vector space.

We will say that this Galois-semiaction is \emph{induced} by the original semirepresentation of $\Gamma$ on $V$.
\end{example}

\begin{definition}\label{definitioninduced}
Let $V$ be a $L$-vector space equipped with a semirepresentation of $\Gamma$. 
Let $X$ be a subscheme of $V$ preserved (normalized) by the induced Galois-semiaction on $V$ understood as a scheme 
(denote this Galois-semiaction by $\psi \colon \Gamma \times V|_K \to V|_K$).
Then we also call the restriction of $\psi$ onto $X$ (formally, $\psi|_{\Gamma \times X|_K}$)
the \emph{induced Galois-semiaction} on $X$.

Similarly, if $X$ is defined by a homogeneous ideal in $S^\bullet (V^*)$, then 
this induced Galois-semiaction can be extended in the obvious way to the projectivization $\mathbf P(X)$.
The resulting Galois-semiaction will also be called the Galois-semiaction on $\mathbf P(X)$ \emph{induced} 
by the semirepresentation of $\Gamma$ on $V$.
\end{definition}
\begin{lemma}\label{inducedisstable}
Let $V$ and $X$ be as in Definition \ref{definitioninduced}, in the second paragraph 
(i.e., suppose $X$ is defined by a homogeneous ideal in $S^\bullet (V^*)$).
Then there exists a $\Gamma$-stable finite affine open covering of $\mathbf P(X)$.
\end{lemma}
\begin{proof}
By Theorem \ref{hilbertninety}
applied to the dual semirepresentation of $\Gamma$ on $V^*$, 
there exists a basis $f_1, \ldots, f_n$ of $V^*$ consisting of $\Gamma$-invariant functions.
Then the 
nonzero
loci of all $f_i$s form the desired covering.
\end{proof}

Now we are ready to define the third category we need
to formulate basic facts of Galois descent theory.
Given a finite Galois extension of 
fields $K \subseteq L$ with Galois group $\Gamma$,
we define the category of 
\emph{stable $L$-schemes with semiaction of $\Gamma$} (notation: $\StSch_{L, \Gamma}$).
Its objects are pairs $(X, \psi)$, where $X$ is an $L$-scheme, 
and $\psi\colon \Gamma \times X|_K \to X|_K$  is a Galois-semiaction such that $X$ admits a $\Gamma$-stable finite affine open covering:
each individual affine open set in this covering has to be preserved (normalized) by $\Gamma$.
The morphisms are morphisms in $\Sch_L$ that become $\Gamma$-equivariant after the restriction of scalars to $K$.

Now recall that if a finite group $\Gamma$ (for now, this is not necessarily the Galois group of any extension of fields) 
acts on a scheme $X$ (now this is going to be a scheme over the smaller field, $K$), 
and there is a $\Gamma$-stable finite affine open covering for this action, 
then the categorical quotient always exists, and can be constructed, for example, as the orbit space of the action.
In a bit more details, 
the points of the \emph{orbit space} $X/\Gamma$ are the $\Gamma$-orbits on (the set of points of) $X$, 
a subset $U$ of $X/\Gamma$ is open if it consists of all orbits in an open ($\Gamma$-invariant) subset $\tilde U$ of $X$, and the algebra of 
functions on such an open set $U$ consists of all $\Gamma$-invariant functions on $\tilde U$.
This gives us a ringed space structure on $X/\Gamma$, and the presence of a $\Gamma$-stable finite affine open covering makes it easy to check that 
with this structure, $X/\Gamma$ actually becomes a scheme over $K$.

So, for a \emph{finite} Galois extension $K \subseteq L$ with group $\Gamma$,
we can define the \emph{Galois descent functor} $\GalDec_{K} \colon \StSch_{L, \Gamma} \to \Sch_K$
as follows: an object $(X, \psi)$ is mapped to the categorical quotient $X / \Gamma$, and 
the morphisms are mapped using the universal property of the categorical quotient.

We can also define the \emph{Galois upgrade functor} $\cdot_{L,\Gamma} \colon \Sch_K \to \StSch_{L, \Gamma}$.
On the objects, it maps a $K$-scheme $Y$ to $(Y_L, \varphi)$, where the semiaction $\varphi$ 
is defined on the affine charts as follows: if $U$ is an open affine chart of $Y$, $\sigma \in \Gamma$, 
then $\varphi (\Gamma, (U_L)|_K) = (U_L)|_K$ (recall that the restriction of scalars does not change the topological space).
And if $f \otimes \lambda \in L[U_L] = K[U] \otimes L$, then 
$(\varphi_{\sigma})^* (f \otimes \lambda) = f \otimes \sigma^{-1}(\lambda)$.
On the morphisms, the Galois upgrade functor is just extension of scalars.
An easy direct check shows that the extension of scalars automatically makes morphisms $\Gamma$-equivariant
with our construction of the $\Gamma$-semiactions on the $L$-schemes.

\begin{remark}\label{galupgforvsp}
Let $U$ be a $K$-vector space. Then $U_{L, \Gamma}$ is canonically isomorphic to $(L \otimes_K U, \psi)$, 
where $\psi$ is the Galois-semiaction on $L \otimes_K U$ understood as a scheme induced (Example \ref{actinducedbyrep})
induced
by the standard semirepresentation (Example \ref{stdrepresuv}) of $\Gamma$ on $L \otimes_K U$.

Similarly, if $X \subseteq U$ (resp.\ $X \subseteq \mathbf P(U)$) is a subscheme, 
then $X_L$ can be canonically embedded into $L \otimes_K U$ (resp.\ $\mathbf P(L \otimes_K U)$), 
and the semiaction on $X_{L,\Gamma}$ is also induced by the standard semirepresentation on $L \otimes_K U$.
\end{remark}

Using the Galois descent and upgrade functors, let us state the main theorem of Galois descent theory.
\begin{theorem}\label{thmgaloisdescent}
Let $K \subseteq L$ be a Galois extension with Galois group $\Gamma$.
The Galois descent and upgrade functors are mutually quasiinverse equivalences of categories
$\Sch_K \leftrightarrow \StSch_{L, \Gamma}$.
\end{theorem}
\begin{proof}
Well-known. 
For a proof
one can see, for example, \cite[\S V.4.20, Proposition 12 and its proof]{serreclass}, 
although the terminology there is a bit different. Instead of actions of $\Gamma$ by semilinear automorphisms, 
the terminology there is based on families of varieties 
(constructed from a single variety by means of all elements of $\Gamma$ in the sense of Definition \ref{serretwist})
and families of morphisms (over $L$, in the standard sense) between these varieties.
The functoriality is not proved there, but it easily follows from 
the explicit construction of $\GalDec_K$ using orbit spaces.
\end{proof}
So, using this theorem, instead of studying schemes over $K$ (they may not have rational points or be otherwise not so nice), 
we can now study varieties over a larger field $L$ 
(which must be a finite Galois extension of $K$, but otherwise we can choose it freely, 
for example so that our schemes over $K$ become nicer when we extend scalars to $L$)
equipped with stable Galois-semiactions.
By Lemma \ref{inducedisstable}, we don't even have to check the stability explicitly
if the semiaction on a projective variety is induced by a semirepresentation.

To work with actions of algebraic groups over $K$ using Theorem \ref{thmgaloisdescent},
we 
need to understand how direct products work in $\Sch_{K,L}$ and in $\StSch_{L, \Gamma}$.
The direct products in $\Sch_L$ and in $\Sch_{K,L}$ are different. 
However, the following lemma shows that direct products from $\Sch_L$ 
are useful in $\Sch_{K,L}$ if we work with semilinear morphisms.
\begin{lemma}\label{propdirectproducts}
Let $K \subseteq L$ be a Galois extension of fields with Galois group $\Gamma$.
Let $X$ and $Y$ be $L$-schemes, let $Z$ be their product in $\Sch_L$,
and let $\pr_1 \in \Mor_{\Sch_L} (Z, X)$ and $\pr_2 \in \Mor_{\Sch_L} (Z, Y)$
be the standard projections.
Then for every $L$-scheme $T$, for every $\sigma \in \Gamma$, and for every two $\sigma$-semilinear morphisms 
$f \colon T \to X$ and $g \colon T \to Y$ there exists a unique $\sigma$-semilinear morphism $h \colon T \to Z$
such that $\pr_1|_K \circ h = f$ and $\pr_2|_K \circ h = g$.
\end{lemma}
\begin{proof}
Consider the $L$-scheme $T^{\sigma^{-1}}$ obtained from $T$ by means of $\sigma^{-1}$ and the $(\sigma^{-1})$-semilinear 
morphism $i:=(\sigma^{-1})_{T^{\sigma^{-1}},T} \colon T^{\sigma^{-1}} \to T$.
The compositions $f \circ i \colon T^{\sigma^{-1}} \to X$ and $g \circ i \colon T^{\sigma^{-1}} \to Y$
are $1_\Gamma$-semilinear, in other words, they are morphisms in $\Sch_L$.
By the universal property of product in $\Sch_L$, they define a morphism $h' \colon T^{\sigma^{-1}} \to Z$.

It follows immediately from the definitions that $i$ is invertible, 
and its inverse $i^{-1} \colon T \to T^{\sigma^{-1}}$ is $\sigma$-semilinear.
So, we can take $h=h \circ i^{-1}$, and this proves the existence of $h$. The uniqueness is proved similarly.
\end{proof}
Suppose, for $K$, $L$, $\Gamma$, $X$, $Y$, $Z$, $\pr_1$, and $\pr_2$ as in the lemma, we have two $\sigma$-semilinear morphisms:
$f \in \Mor_{\Sch_{K,L}} (A, X)$ and $g \in \Mor_{\Sch_{K,L}} (B, Y)$. Let $C$ be the product of $A$ and $B$ in $\Sch_L$, 
and let $\pr_1' \in \Mor_{\Sch_L} (C, A)$ and $\pr_2' \in \Mor_{\Sch_L} (C, B)$
be the standard projections. In this case we will denote by $f \times g \in \Mor_{\Sch_{K,L}} (C,Z)$ the 
unique $\sigma$-semilinear morphism such that $\pr_1|_K \circ (f \times g) = f \circ \pr_1'|_K$ and
$\pr_2|_K \circ (f \times g) = g \circ \pr_2'|_K$.
Informally speaking, this is a straightforward way to build a morphism $A \times B \to X \times Y$ 
out of morphisms $A \to X$ and $B \to Y$.

After we have this lemma, it is easy to construct a Galois-semiaction on a product of two $L$-schemes $X$ and $Y$
out of two semiactions on $X$ and on $Y$. 
Precisely, if $\psi_1 \colon \Gamma \times X|_K \to X|_K$ and $\psi_2 \colon \Gamma \times Y|_K \to Y|_K$ 
are two Galois-semiactions, then the new semiaction on $Z=X\times Y$ (the product in $\Sch_L$), 
which we will call \emph{the product of semiactions} and denote $\psi_1 \times \psi_2$, is defined as follows:
$(\psi_1 \times \psi_2)_{\sigma} = (\psi_{1,\sigma}) \times (\psi_{2,\sigma})$.
Then a direct check shows that $(Z, \psi_1 \times \psi_2)$ is the product of $(X, \psi_1)$ and $(Y, \psi_2)$ in $\StSch_{L, \Gamma}$.

Using this description of products, we can say, for example, the following. 
\begin{example}
Let $K$, $L$, and $\Gamma$ be as above, and let $G$ be an algebraic group over $K$.
Denote by $m$ the multiplication morphism $G \times G \to G$ 
and by $\psi_1$ the Galois-semiaction such that $G_{L,\Gamma} = (G_L, \psi_1)$.
Then $m_L$ is $\Gamma$-equivariant for the semiactions $\psi_1 \times \psi_1$ and~$\psi_1$.

Moreover, let $X$ be a scheme over $K$, let $\varphi \colon G \times X \to X$ be an action, and 
let $\psi_2$ be the Galois-semiactions such that $X_{L,\Gamma} = (X_L, \psi_2)$. 
Then $\varphi_L$ is $\Gamma$-equivariant for the 
semiactions
$\psi_1 \times \psi_2$ and~$\psi_2$.
\end{example}

This finishes the part of theory of Galois descent that we need.

\section{Preparation 2: Isomorphism of Picard groups under scalar extension}\label{sectionpicard}
To prove Theorem \ref{maintheorem}, we have to, roughly speaking, study 
the behavior of quotients of torsors modulo Borel subgroups under scalar expansion.
First, let us
define the quotient of a torsor modulo a Borel subgroup.
The definition we are going to use is not very intuitive, but it is used in papers on canonical dimension (for example, in \cite{karpenkomerkurjevadvances}).
\begin{definition}\label{defemodb}
Let $G$ be a semisimple split simply connected algebraic group over a field $K$, let $B$ be a Borel subgroup, and let $E$ be a $G$-torsor. 
The \emph{quotient of the torsor modulo the Borel subgroup} (notation: $E/B$) is the categorical quotient 
(see \cite[Definition 0.5]{mumford}; ``categorical'' is in the category of all separated schemes of finite type
over $K$)
of $E \times G/B$ modulo the diagonal action of $G$.
\end{definition}
In fact, it can be proved that such a quotient is also a categorical quotient of $E$ modulo $B$, but we will not use this.
The existence of such a categorical quotient $(E \times G/B)/G$ is known,
is stated in \cite[Proposition 12.2]{florence},
and can be proved using Galois descent theory.
We will need an explicit construction for $E/B$, and we will recall it below.
It is known that such a quotient $E/B$ is smooth, 
absolutely integral, 
and projective.

Given this definition, we can now say precisely that the first 
step in proving Theorem \ref{maintheorem}
is to prove the following proposition.
\begin{proposition}\label{propositionpicards}
Let $G$ be a semisimple split simply connected algebraic group over a field $K$, let $B$ be a Borel subgroup, and let $E$ be a $G$-torsor. 
Let $K'$ be an extension of $K$.
Then the map of Picard groups induced by field extension $\Pic (E/B) \to \Pic ((E/B)_{K'})$ is an isomorphism.
\end{proposition}

%
%

And our first step in the proof of Proposition \ref{propositionpicards} itself will be to 
recall a more explicit construction of $E/B$ (and of $G/B$ as well).

First, recall also that a torsor is called \emph{trivial} if it has a rational point.
If $(E, \varphi)$ is a trivial torsor of an algebraic group $G$, and if $e$ is a rational point of $E$, 
then we denote the map $\varphi|_{G \times \{e\}} \colon G \to E$ by $\triv_e$ (``torsorization of the group''). 
One can easily check that
this is an isomorphism
(see Appendix \ref{picardcalculations}, Lemma \ref{triveiso} for details; 
also, see Appendix \ref{picardcalculations} for other calculations omitted in this section).
We keep the notation $\triv_e$ until the end of the text. 
Also, for simplicity of further formulas, we denote $\trivinv_e:= \triv_e^{-1} \colon E \to G$ (``detorsorization''), and
we denote the composition of the inversion map $\inv$ on $G$ and $\trivinv_e$ by $\trivinvinv_e$: $\trivinvinv_e=\inv \circ \trivinv_e$
(``inversion-detorsorization'').
Finally, if it is clear which rational point $e$ of $E$ we are talking about, we simply write $\triv$, $\trivinv$, and $\trivinvinv$
instead of $\triv_e$, $\trivinv_e$, and $\trivinvinv_e$, respectively.

Then, we need a lemma.
\begin{lemma}
Let $(E, \varphi)$ be a torsor of a \emph{smooth} algebraic group $G$ over a field $K$. Then there exists a finite Galois extension $L$ of $K$
such that $(E_L, \varphi_L)$ is a trivial $G_L$-torsor.
\end{lemma}
\begin{proof}[Idea of the proof]
It is easy to see that 
$E$ is smooth. 
Smooth schemes obtain a rational point after scalar extension to a separable closure (\cite[Prop. 3.2.20]{qingliu}).
\end{proof}
So, instead of studying a torsor without rational points, 
we can do a finite Galois extension of scalars and study a torsor 
with a rational point and with a 
compatible Galois-semiaction.

From now on, we fix until the end of this section:
a split semisimple 
simply connected 
algebraic group $G$ over a field $K$, a Borel subgroup $B$ of $G$, a maximal torus $T$ of $G$ contained in $B$, 
a $G$-torsor $(E, \varphi)$, 
a finite Galois extension $L$ of $K$ such that $E_L$ has a rational point, 
and a rational point $e \in E_L$. Denote $\Gamma = \Gal (L/K)$.
It is known that split semisimple groups stay split semisimple after base field extension, and that their Borel subgroups also stay Borel.

Denote the action map $G_L \times (G/B)_L \to (G/B)_L$ by $\xi$, 
and for each individual element (rational point) $g$ of $G_L$, denote by $\xi_g$ 
the action of this element $g$ on $(G/B)_L$ (in other words, $\xi_g = \xi|_{\{g\} \times (G/B)_L}$).
Denote by $\psi_1$, $\overline{\psi}_1$, $\psi_2$ 
the semiactions such that $G_{L,\Gamma} = (G_L, \psi_1)$, 
$(G/B)_{L,\Gamma} = ((G/B)_L, \overline{\psi}_1)$, 
and $E_{L,\Gamma} = (E_L, \psi_2)$.

For a strongly dominant weight $\lambda$ of $G$, denote the corresponding representation of $G$ by $V_\lambda$. 
Recall that
if $v_\lambda \in V_\lambda$ is a highest weight vector, then 
the stabilizer
of $\ell = \Span (v_\lambda)$ is $B$, and that $G/B$ can be constructed as $G \ell \subseteq \mathbf P (V_\lambda)$ 
(this orbit is known to be closed). So, we have a separate construction for $G/B$ for each strongly dominant weight.
Let us also note right away that this construction
obviously commutes with the field extension, 
so $(G/B)_L$ can be constructed as $G_L \Span (1 \otimes v_\lambda) \subseteq \mathbf P (L \otimes_K V_\lambda)$.
So, the action $\overline{\psi}_1$ is induced by the standard semirepresentation of $\Gamma$ on $L \otimes_K V_\lambda$.

Like 
with $G/B$, we will also have a separate construction for $E/B$ for each strongly dominant weight $\lambda$ of $G$.
We are going to construct $E/B$ as the Galois descent of $(G/B)_L$ equipped with a specific Galois-semiaction
(most likely different from $\overline{\psi}_1$). 
First, 
denote by $p$ the following map from $E_L \times (G/B)_L$ to $(G/B)_L$ (here, $\id_{(G/B)_L}$ is the identity map):
\begin{equation*}
p = \xi \circ (\trivinvinv_e \times \id_{(G/B)_L})
\end{equation*}
In other words, we first isomorphically map $E_L$ to $G_L$, then invert $G_L$ (during these maps, $(G/B)_L$ stays untouched), 
and then we act by $G_L$ on $(G/B)_L$.
\begin{lemma}\label{gblisquot}
The variety $(G/B)_L$ together with the map $p$ is a categorical quotient of $E_L \times (G/B)_L$ modulo the diagonal action of $G_L$.
\end{lemma}
\begin{proof}
The $G_L$-invariance of $p$ is a direct computation. The universal property is an easy diagram chase.
For more details, see Appendix \ref{picardcalculations}.
\end{proof}
However, the map $p$ is not equivariant for the semiactions $\psi_2 \times \overline{\psi}_1$ and $\overline{\psi}_1$. 
Let us introduce a new semiaction $\overline{\psi}_2$ on $(G/B)_L$. Namely, for each $\sigma \in \Gamma$, set
\begin{equation}\label{newgaction}
\overline{\psi}_{2,\sigma} = 
\xi_{\trivinvinv_e (\psi_2(\sigma, e))} \circ \overline{\psi}_{1,\sigma}.
\end{equation}
\begin{lemma}\label{actonebisrepres}
The formulas (\ref{newgaction}) indeed define a semiaction on $(G/B)_L$.
Moreover, for each strongly dominant weight $\lambda$ and for the corresponding embedding $(G/B)_L \hookrightarrow \mathbf P (L \otimes_K V_\lambda)$,
$\overline{\psi}_2$
is induced by a semirepresentation of $\Gamma$ on $L \otimes_K V_\lambda$.
In particular, there exists a stable finite affine open covering of $(G/B)_L$ for $\overline{\psi}_2$.
\end{lemma}
\begin{proof}
Denote the standard semirepresentation of $\Gamma$ on $L \otimes_K V_\lambda$ by $\widetilde{\psi}_1$.
Denote the action of an element (a rational point) $g$ of $G_L$ on $L \otimes_K V_\lambda$
by $\Xi_g$.

For each $\sigma \in \Gamma$, denote 
\begin{equation}\label{newgrep}
\widetilde{\psi}_{2, \sigma} = \Xi_{\trivinvinv_e(\psi_2(\sigma, e))} \circ \widetilde{\psi}_{1, \sigma }.
\end{equation}
This is a composition of a $\sigma$-semilinear map from the vector space $L \otimes_K V_\lambda$
to itself and a linear one, so it is also a $\sigma$-semilinear map of vector spaces.
One more direct computation, this time using the $\Gamma$-equivariance of the 
representation map $G_L \times (L \otimes_K V_\lambda) \to L \otimes_K V_\lambda$
and of the action map $\varphi_L$ shows that 
all these maps together 
form a semirepresentation of $\Gamma$ on $L \otimes_K V_\lambda$
(see Appendix \ref{picardcalculations}, Lemma \ref{psitwoissemirep} for more details).
Then it is clear from the formulas (\ref{newgaction}) and (\ref{newgrep})
that this semirepresentation induces the semiaction $\overline{\psi}_2$.
The existence of a stable finite affine open covering follows from Lemma \ref{inducedisstable}.
\end{proof}
\begin{sideremark}
In terms of Galois cohomology (which we 
will not use here,
but details can be found in \cite[Section 5.1]{serre}), 
the formula for $\overline{\psi}_2$ can be formulated as 
``$\overline{\psi}_2$ is obtained from $\overline{\psi}_1$
by twisting by the cocycle 
$(\sigma \mapsto \xi_{\trivinvinv_e (\psi_2(\sigma, e))}) \in H^1 (\Gamma, \Aut ((G/B)_L))$''.
\end{sideremark}
\begin{lemma}\label{pequivar}
The map $p$ is equivariant for the semiactions $\psi_2 \times \overline{\psi}_1$ and $\overline{\psi}_2$. 

Moreover, let $q \colon (E_L \times (G/B)_L, \psi_2 \times \overline{\psi}_1) \to (X, \psi_3)$
be another $G_L$-invariant and $\Gamma$-equivariant morphism 
(where $X$ is an arbitrary $L$-scheme with a Galois-semiaction $\psi_3$).
Then the unique map $r \colon (G/B)_L \to X$ from the universal property of the categorical quotient
is actually $\Gamma$-equivariant for the semiactions $\overline{\psi}_2$ and $\psi_3$.
\end{lemma}
\begin{proof}
Direct computation. The first statement again uses the $\Gamma$-equivariance of the maps $\varphi_L$ 
(for the semiactions $\psi_1 \times \psi_2$ and $\psi_2$) 
and $\xi$
(for the semiactions $\psi_1 \times \overline{\psi}_1$ and $\overline{\psi}_1$).
The second statement uses the uniqueness of the map in the universal property of a categorical quotient.
See Appendix \ref{picardcalculations} for more details.
\end{proof}
\begin{proposition}\label{ebconstruction}
In terms of the notation introduced above in this section,
the scheme $\GalDec_K((G/B)_L, \overline{\psi}_2)$ with the map $\GalDec_K(p)$
is a categorical quotient of $E \times (G/B)$ modulo the diagonal action of $G$.
Therefore, it can be used as $E/B$, and $(E/B)_L$ then becomes isomorphic to $(G/B)_L$.
\end{proposition}
\begin{proof}
Follows from Lemmas \ref{actonebisrepres} (for stability), \ref{gblisquot}, and \ref{pequivar}. Also uses Theorem \ref{thmgaloisdescent}.
\end{proof}
\begin{remark}\label{etrivebisogb}
It follows from this construction that if $E$ itself is trivial, then we can take $L=K$ and see that $E/B$ is isomorphic to $G/B$.
\end{remark}

Now, after we have recalled an explicit construction of $E/B$, 
there are two (most easy) ways to prove Proposition \ref{propositionpicards}: one uses representation theory and Galois descent further, 
and the other uses \'etale sheaves.
In this section, we continue with the proof using representation theory, 
and the proof using \'etale sheaves will be outlined in Appendix \ref{appendixetale} 
(although that proof will still use a previously known result about Picard groups 
we quote below in this section, Lemma \ref{picardexact}).

For the proof using representation theory,
let us first prove the surjectivity in 
Proposition \ref{propositionpicards} for $K' = L$ (the field we have fixed).
%
We will need
the following explicit description of the Picard group of a flag variety.
\begin{theorem}\label{theorempicflags}
For $K$, $L$, $G$, and $B$
as above, denote the weight lattice of $G_L$ by $\Lambda$.
For each strongly dominant weight $\lambda \in \Lambda$, denote by $\mathcal L_\lambda$ 
the pullback of the anticanonical bundle under the embedding 
$G_L/B_L \hookrightarrow \mathbf P (L \otimes_K V_\lambda)$
described above. Then:
\begin{enumerate}
\item The notation $\mathcal L_\lambda$
and the map $\lambda \mapsto \mathcal L_\lambda$ (which we have so far defined for strongly dominant weights $\lambda$ only)
can be extended to a group homomorphism $\Lambda \to \Pic (G_L/B_L)$. Moreover, this group homomorphism
is actually an isomorphism.
\item In terms of this notation, if $\lambda_i$ is the $i$th fundamental weight, then the vanishing locus of 
(any) global section of $\mathcal L_{\lambda_i}$ is linearly equivalent to $(D_i)_L$, where $D_i$ is the divisor described in the Introduction.
\end{enumerate}
\end{theorem}
\begin{proof}
Well-known.
\end{proof}
Also, we will need the following lemma about hyperplane sections in projective varieties in general.
\begin{lemma}\label{represmeansdivdescent}
Let $L/K$ be a finite Galois extension of fields (unlike elsewhere in this section, 
generally speaking, they don't have to be related to any algebraic groups or torsors), $\Gamma=\Gal(L/K)$.
Let $V$ be an $L$-vector space with a semirepresentation of $\Gamma$, 
and let $X \subseteq \mathbf P (V)$ be an integral subscheme 
with the induced Galois-semiaction $\psi$. 
Let $D \in \CH^1 (X)$ be the class of (any) hyperplane section of $X$.

Then, after the identification $(X, \psi) \cong (\GalDec_K (X,\psi))_{L,\Gamma}$, 
$D$ belongs to the image of the scalar extension map $\CH^1(\GalDec_K (X,\psi)) \to \CH^1(X)$.
\end{lemma}
\begin{proof}
Consider the dual semirepresentation of $\Gamma$ on $V^*$. 
By Theorem \ref{hilbertninety}, there exists a nonzero linear function $f \in (V^*)^\Gamma$. 
Then the vanishing locus of $f$ in $X$ is a $\Gamma$-invariant hyperplane section. Denote this hyperplane section by $Y$.

It follows from the explicit construction of $\GalDec_K$ using orbit spaces that the Galois descent of 
the embedding of $Y$ into $X$ is still an embedding of a closed subscheme.
By Theorem \ref{thmgaloisdescent},
this subscheme 
becomes $Y$ after the extension of scalars back to $L$.
\end{proof}
\begin{proposition}\label{propositionpicardssurjective}
Let $E$, $B$, and $L$ be as above.
Then the map of Picard groups induced by field extension $\Pic (E/B) \to \Pic ((E/B)_L)$ is surjective.
\end{proposition}
\begin{proof}
Follows from Proposition \ref{ebconstruction}, Theorem \ref{theorempicflags}, Lemma \ref{actonebisrepres}, and Lemma \ref{represmeansdivdescent}.

More accurately, we also need the fact that
for any smooth and absolutely connected scheme $X$, 
the isomorphism $\Pic (X) \to \CH^1 (X)$ commutes with extension of scalars
(this is well-known), 
and the fact that the construction of $G/B$ also commutes with extension of scalars (this follows directly from the construction, as we have already mentioned).
\end{proof}

Now, after we have proved the surjectivity in this case, we will need 
to recall a well-known result about Picard and Brauer groups.
First, note that for any Galois-semiaction on an integral scheme $Y$
there is a straightforward way to extend 
this semiaction to an action on the set of open subsets of $Y$, on 
the field of rational functions on $Y$,
and therefore on the Picard group of $Y$.

With these remarks,
let us state the result about Picard and Brauer groups.
For any two fields $K_1 \subseteq L_1$, denote $\Br_{L_1}(K_1) = \ker (\cdot \otimes_{K_1} L_1 \colon \Br(K_1) \to \Br(L_1))$.
\begin{lemma}\label{picardexact}
Let $X$ be a 
smooth
integral scheme over a field $K_1$ 
without nonconstant invertible functions.
Let $L_1$ be a finite Galois extension of $K_1$ such that 
$X_{L_1}$ is still integral and the only invertible functions on $X_{L_1}$ are still constants.
Let $\Gamma' = \Gal (L_1/K_1)$.
Then: 
\begin{enumerate}
\item\label{picardinsidegamma}
The image of the map $\cdot_{L_1} \colon \Pic (X) \to \Pic(X_{L_1})$ is contained in $\Pic (X_{L_1})^{\Gamma'}$.
\item\label{picardtrueexact}
There is an exact sequence
\begin{equation*}
0 \to \Pic(X) \xrightarrow{\cdot_{L_1}} \Pic(X_{L_1})^{\Gamma'} \to \Br_{L_1} (K_1) \xrightarrow{\cdot \otimes K_1(X)} \Br (K_1(X)).
\end{equation*}
\end{enumerate}
\end{lemma}
\begin{proof}
Well-known (see, for example, \cite[Proof of theorem 3.1]{coilliot}).
Follows from exact sequences 
\begin{equation*}
1 \to {L_1}^* \to L_1(X_{L_1})^* \to L_1(X_{L_1})^*/{L_1}^* \to 1
\end{equation*}
and 
\begin{equation*}
1 \to {L_1}(X_{L_1})^* / {L_1}^* \to \Div (X_{L_1}) \to \Pic (X_{L_1}) \to 1.
\end{equation*}
\end{proof}

Now we are ready to prove Proposition \ref{propositionpicards} in the whole generality.
\begin{lemma}\label{proptruefortrivial}
Proposition \ref{propositionpicards} is true if $E$ is a trivial torsor.
\end{lemma}
\begin{proof}
Follows from 
Remark \ref{etrivebisogb}
and the explicit description of $\Pic (G/B)$ (like Theorem \ref{theorempicflags}, but over an arbitrary field instead of $L$).
\end{proof}
\begin{lemma}\label{proptrueforfixedgalois}
Proposition \ref{propositionpicards} is true when $K'$ equals $L$ (the field we fixed earlier in this section).
\end{lemma}
\begin{proof}
The injectivity follows\footnote{The idea of using the exact sequence of Brauer and Picard groups to prove the isomorphism between Picard groups 
is present in \cite[Proof of Theorem 1.4]{karpenkoduke}.}
from Lemma \ref{picardexact} (\ref{picardtrueexact}) for $K_1 = K$, $L_1=L$.
The surjectivity is Proposition \ref{propositionpicardssurjective}.
\end{proof}
\begin{proof}[Idea of proof of Proposition \ref{propositionpicards} in the general case]
%
We omit the details regarding commutativity of the diagrams of Picard groups for consecutive field extensions.
First, 
note that
the proposition for $K'$ containing $L$ follows from Lemma \ref{proptruefortrivial} for the torsor $E_L$ and for the extension $K'/L$
(and from Lemma \ref{proptrueforfixedgalois}).

Then, for a completely arbitrary $K'$ containing $K$, we first find a finite Galois extension $L'$ of $K'$ for the $G_{K'}$-torsor $E_{K'}$
in the same way as we found and fixed $L$ for $K$, $G$, and $E$, so that $(E_{K'})_{L'}$ has a rational point. 
Since $L$ is a finite Galois extension of $K$, 
we can construct a field $L''$ admitting embeddings of $L$ and of $L'$.
By the previous step for $E_{K'}$ instead of $E$, $\Pic ((E/B)_{K'}) \cong \Pic ((E/B)_{L''})$.
By the previous step for the original $E$, 
$\Pic (E/B) \cong \Pic ((E/B)_{L''})$. 
Therefore, $\Pic (E/B) \to \Pic ((E/B)_{K'})$ is an isomorphism.
\end{proof}

\section{Estimate of canonical dimension}\label{sectionlaststeps}

The next steps of the proof of Theorem \ref{maintheorem} follow the idea of proof of \cite[Proposition 5.1]{karpenkoduke}.

First, we will need a result from \cite{karpenkomerkurjevadvances}.
To formulate it, let us start with recalling 
a definition from \cite{karpenkomerkurjevadvances}. Let $X$ be a scheme over an arbitrary field $F$. 
The \emph{determination function associated with $X$} (see \cite[Section 2]{karpenkomerkurjevadvances}) 
is the following functor from the category of all fields containing $F$ to the category 
consisting of $\varnothing$ and a fixed one-element set $\{0\}$: A field $F_1$ is mapped to $\{0\}$ if and only if $X_{F_1}$ has a rational point, 
otherwise $F_1 \mapsto \varnothing$.

Also, 
recall that an 
algebraic group $H$ over $K$
is called \emph{special} if all torsors of all groups $H_{K'}$, where $K'$ is a field extension of $K$, are trivial.
It is known (see, for example, \cite[Section 3 and Theorem 2.1]{karpenkogensplit}) that 
the Borel subgroup of a split reductive group
is special.

Now, with these two definitions, 
let us quote the result from \cite{karpenkomerkurjevadvances} and immediately get two corollaries.
\begin{lemma}[{\cite[Lemma 6.5]{karpenkomerkurjevadvances}}]
Let $P$ is a subgroup of an algebraic group $G$. Suppose that $P$ us a special group. Let $E$ be a $G$-torsor.
Then the determination functions of the varieties $E$ and $E/P$ coincide.\qed
\end{lemma}
\begin{corollary}\label{detfuncscoincide}
Let $G$ be a semisimple split simply connected algebraic group over a field $K$, let $E$ be a $G$-torsor, and let $B$ be a Borel subgroup of $G$.
Then for any field extension $K'/K$, 
$E_{K'}$ has a rational point if and only if $(E/B)_{K'}$ has a rational point.\qed
\end{corollary}

\begin{corollary}\label{cdeebequal}
Under the notation of 
Corollary \ref{detfuncscoincide},
$\cd (E) = \cd(E/B)$.\qed
\end{corollary}

Then, we will need a well-known fact about the Chow ring of a smooth scheme.
\begin{proposition}\label{chowmorphism}
Let $X$ be a smooth scheme over a field $K$, and let $L$ be an extension of $K$. 
The map of Chow rings $\CH_L \colon \CH (X) \to \CH (X_L)$, $[Y] \mapsto [Y_L]$ for each 
integral\footnote{We will not need this, but the map defined this way 
actually maps the class of any subscheme $Y$ of $X$ to $[Y_L] \in \CH (X_L)$.}
subscheme $Y$ of $X$ is well-defined and is a morphism of rings.

The isomorphism $\Pic (X) \to \CH^1 (X)$ commutes with extension of scalars.
\end{proposition}
\begin{proof}
Well-known.
\end{proof}
We will also need the following theorem. It is stated in \cite[Theorem 2.3]{karpenkoduke} and 
follows from \cite[Corollary 12.2]{fulton}, the preceding commutative diagram, 
and the definition of \emph{distinguished varieties} in \cite{fulton}. 
More precisely, this definition implies that in the particular case of the commutative diagram, 
the distinguished varieties are subvarieties of the intersection of supports of the cycles.
Recall that a cycle (a formal linear combination of irreducible subvarieties) is called \emph{nonnegative} 
if the coefficients in this linear combination are nonnegative, 
and an element of the Chow ring is called \emph{nonnegative} if it can be represented by a nonnegative cycle.
\begin{theorem}\label{nonnegativeintersection}
Let $X$ be a smooth scheme over an arbitrary field $K$ such that the tangent bundle is generated by global sections.
Let $\alpha$ and $\beta$ be 
nonnegative 
elements of $\CH (X)$.
If $\alpha$ (resp.\ $\beta$) is represented by a 
nonnegative cycle with support on $A \subseteq X$ (resp.\ $B \subseteq X$), 
then $\alpha \beta$ 
can be represented by a 
nonnegative cycle with support on $A \cap B$.\qed
\end{theorem}
We need two more facts from \cite{karpenkoduke}:
\begin{lemma}[{\cite[Remark 2.4]{karpenkoduke}}]\label{ebisgbgs}
Let $G$ be a split simple simply connected algebraic group over an arbitrary field $K$, let $B$ be a Borel subgroup of $G$, 
let $E$ be a $G$-torsor. Then the tangent bundle of $E/B$ is generated by global sections.\qed
\end{lemma}
\begin{lemma}[{\cite[Corollary 2.2]{karpenkoduke}}]\label{scalarextpreservesnonneg}
Let $X$ be a smooth absolutely irreducible scheme over an arbitrary field $K$, and let $L$ 
be an extension of $K$.
Let $\alpha \in \CH^1(X)$.
If 
$\CH_{L}(\alpha) \in \CH^1 (X_L)$ 
is nonnegative, then $\alpha \in \CH^1(X)$ is nonnegative.\qed
\end{lemma}
The following proposition is like Proposition 5.1 in \cite{karpenkoduke}, but in a different situation.
\begin{proposition}\label{subvarexists}
Let $G$ be a semisimple split simply connected algebraic group over 
an arbitrary field $F$.
Let $B$ be a Borel subgroup, and let $D_1, \ldots, D_r \subset G/B$ be the Schubert divisors.
Suppose that the product $[D_1]^{n_1} \ldots [D_r]^{n_r}$ is multiplicity-free.

Let $K$ be a field extension of $F$, and let $E$ be a $G_K$-torsor.
Then 
there exists a 
closed integral
subscheme $Y$ of $E/B_K$ of codimension $n_1 + \ldots + n_r$
such that $Y_{K(E/B_K)}$ has a rational point.
\end{proposition}
\begin{proof}
Denote $X = E/B_K$ and $L 
= K(X)$. 
Write
\begin{equation*}
[D_1]^{n_1}[D_2]^{n_2}\dots [D_r]^{n_r}=\sum c_{w,n_1,\ldots,n_r}[Z^w].
\end{equation*}
Fix an element $v \in W$ such that $c_{v,n_1,\ldots,n_r} = 1$. Set $v'=v w_0$.
Then it follows from 
\cite[\S 3.3, Proposition 1a]{demazure}
that 
$[D_1]^{n_1} \ldots [D_r]^{n_r}[Z^{v'}] = [\pt]$.
By Proposition \ref{chowmorphism},
we have 
$[(D_1)_L]^{n_1} \ldots [(D_r)_L]^{n_r} [(Z^{v'})_L] = [\pt] \in \CH((G/B)_L)$.

It is easy to see that $X_L$ has a rational point. 
By 
Corollary \ref{detfuncscoincide},
$E_L$ also has a rational point. 
Then by 
Remark \ref{etrivebisogb},
$X_L$ is isomorphic to $(G/B)_L$.
Fix one such isomorphism (it depends on the choice of a rational point of $E_L$) and denote it by $f \colon X_L \to (G/B)_L$.

Denote 
the composition $f_* \circ \CH_L \colon \CH (X) \to \CH (X_L) \to \CH ((G/B)_L)$
by $g$.
Denote $g_1 = g|_{\CH^1 (X)}$.
By Proposition \ref{propositionpicards} (and by Proposition \ref{chowmorphism}), 
$g_1$ is an isomorphism between $\CH^1 (X)$ and $\CH^1 ((G/B)_L)$
For each $i$ ($1 \le i \le r$), 
denote $\alpha_i = g_1^{-1} ((D_1)_L) \in \CH^1 (X)$.
By Lemma \ref{scalarextpreservesnonneg}, these are nonnegative classes (although we don't claim that 
each $\alpha_i$ is representable by a single irreducible and reduced divisor).

By Theorem \ref{nonnegativeintersection}, 
the class $\alpha_1^{n_1} \ldots \alpha_r^{n_r}$ is nonnegative.
Choose integral subvarieties $Y_i \subseteq X$ of codimension $n_1 + \ldots + n_r$
such that $\alpha_1^{n_1} \ldots \alpha_r^{n_r}$ can be written as their linear combination 
with nonnegative coefficients. Denote these coefficients by $k_i \ge 0$:
\begin{equation*}
\alpha_1^{n_1} \ldots \alpha_r^{n_r} = \sum k_i [Y_i].
\end{equation*}

It is clear from the definitions that for each $i$, $g([Y_i])$ is a linear combination 
of the irreducible components of $f((Y_i)_L)$ with nonnegative coefficients.
Since $g$ is a morphism of rings (Proposition \ref{chowmorphism}), we have
\begin{equation*}
g\left(\sum k_i [Y_i]\right)[(Z^{v'})_L] = [(D_1)_L]^{n_1} \ldots [(D_r)_L]^{n_r} [(Z^{v'})_L] = [\pt].
\end{equation*}
On the other hand,
$g(\sum k_i [Y_i])[(Z^{v'})_L] = \sum (k_i g([Y_i]) [(Z^{v'})_L])$, and by Theorem \ref{nonnegativeintersection}, 
each $g([Y_i]) [(Z^{v'})_L]$ is (can be written as) a linear combination
of (reduced) 0-dimensional subvarieties (i.~e. closed points) of $f((Y_i)_L) \cap (Z^{v'})_L$
with nonnegative coefficients.

So, a rational point of $(G/B)_L$ is equivalent in the Chow ring to a linear combination of some closed points 
with nonnegative coefficients. Then it follows from the well-definedness of the degree map 
$\CH^{\dim (G/B)} ((G/B)_L) \to \ZZ$ (see \cite[Definition 1.4]{fulton}) that the linear combination 
actually consists of just one point with coefficient 1, and this point is rational. 
Recall that this was a point in some intersection $f((Y_i)_L) \cap (Z^{v'})_L$. 
In particular, we see that for one of the schemes $Y_i$, $(Y_i)_L$ has a rational point, and we can set $Y=Y_i$.

(We don't need this, but 
for this index $i$ we also get $k_i = 1$, and for all other indices $i$ 
we get $g([Y_i]) [(Z_{v'})_L]=0$ or $k_i=0$.)
\end{proof}

\begin{proof}[(Last steps of the) Proof of Theorem \ref{maintheorem}]
Let $F$ be the base field of $G$.



Let $K$ be an extension of $F$, let $E$ be a $G_K$-torsor. Denote $L = K(E/B_K)$.
By Proposition \ref{subvarexists}, there exists a subscheme $Y \subseteq E/B_K$ of codimension $n_1 + \ldots + n_r$
such that $Y_L$ has a rational point. 
Then it follows easily that there exists a rational map $E/B_K \dashrightarrow Y$.
As we already mentioned in the Introduction, since $E/B_K$ is smooth and projective, 
by \cite[Theorem 1.7]{merkurjevcontmath}, $\cd (E/B_K)$ can be computed as the minimal dimension 
of a subscheme $Y' \subseteq E/B_K$ admitting a rational map $E/B_K \dashrightarrow Y'$.
Hence, $\cd(E/B_K) \le \dim Y = \dim (G/B) - n_1 - \ldots - n_r$.

Therefore, we get 
$\cd(E) \le \dim (G/B) - n_1 - \ldots - n_r$ by Corollary \ref{cdeebequal} and 
$\cdg (G) \le \dim (G/B) - n_1 - \ldots - n_r$ by the definition of $\cdg (G)$.
\end{proof}

\appendix

\makeatletter
\def\@seccntformat#1{\appendixname\ \csname the#1\endcsname. }
\makeatother

\section{Detailed calculations for proofs in Section \ref{sectionpicard}}\label{picardcalculations}

In this section, we will do several calculations with morphisms between schemes over non-algebraically closed fields.
To make these calculations convenient, especially when we work with morphisms 
between products of schemes, let us say a few words about notations for such morphisms.
Recall that a corollary of Yoneda lemma says that if, for any object $X$ in a category, $h_X$ denotes the functor $\Hom (-,X)$, 
then for each pair of objects $X$ and $Y$ there is a (functorial in $X$ and $Y$) bijection between $\Hom (X, Y)$ and $\Hom (h_X, h_Y)$.
More explicitly, this bijection sends a morphism $f \colon X \to Y$ to the morphism of functors 
$f \circ {} \colon h_X \to h_Y$, which, for each object $T$, maps $x \colon T \to X$ to $f \circ x$.
The inverse bijection sends a morphism of functors $f \colon h_X \to h_Y$ to $f (\id_X) \in h_Y (X) = \Hom (X, Y)$.

So, first, when we need to combine several morphisms of schemes, we will view each morphism $f \colon X \to Y$ 
at the same time as a map $\Hom (T, X) \to \Hom (T, Y)$ for each scheme (``test scheme'') $T$, and we will also denote this
map $\Hom (T, X) \to \Hom (T, Y)$ by $f$. In other words, if $x \colon T \to X$ is a morphism, we write $f(x)$ instead of $f \circ x$.
Given two morphisms of schemes, $f \colon X \to Y$ and $g \colon Y \to Z$, 
we can denote their composition as follows: for each test scheme $T$, each morphism $x \colon T \to X$ is mapped to $g (f(x)):=g \circ f \circ x$.
Second, instead of writing morphisms to products of two (or more) schemes, we will write pairs (or tuples) consisting of 
compositions of these morphisms with projections onto the factors. For example, if $x \colon T \to X$ and $y \colon T \to Y$
are two arbitrary morphisms, then $x, y$ denotes the morphism $T \to X \times Y$ whose compositions with projections 
are $x$ and $y$. And if $f \colon X \times Y \to Z$ is another morphism, then $f (x,y) \colon T \to Z$ is the composition 
of $x,y$ and $f$. Third, given a rational point $x_0$ of a scheme $X$, we also, for any test scheme $T$, denote by $x_0$ 
the morphism $T \to X$ that maps everything to $x_0$.

Fourth, if a scheme $G$ actually is (or is equipped with the structure of) an algebraic group, 
then we denote the compositions of maps to $G$ with the product map $m \colon G \times G \to G$ and the inversion map $\inv \colon G \to G$
exactly as we denote the products and the inverses of rational points: 
if $T$ is an arbitrary test scheme, and $g_1 \colon T \to G$ and $g_2 \colon T \to G$ are arbitrary maps, 
then we write $g_1g_2$ instead of $m (g_1, g_2)$ and $g_1^{-1}$ instead of $\inv (g_1)$. 
Also, since $1_G \in G$ is a rational point, as we said before, $1_G$ will also denote the map $T \to G$
that maps everything to $1_G$. 
And if $G$ acts a scheme $X$ via a morphism $\varphi \colon G \times X \to X$
(and there is no ambiguity of which action of $G$ on $X$ we are talking about), then we also use multiplicative notation for 
the compositions of maps to $G \times X$ and $\varphi$: if we have maps $g \colon T \to G$ and $x \colon T \to X$, 
then we write $gx$ instead of $\varphi(g,x)$.

Finally, suppose we have schemes $X_1, \ldots, X_n$ and $Y_1, \ldots, Y_m$ and morphisms between some of them, and we want 
to define a morphism $X_1 \times \ldots \times X_n \to Y_1 \times \ldots \times Y_m$. 
To do this, we fix an arbitrary scheme (``test scheme'') $T$ and arbitrary morphisms $x_i \colon T \to X_i$, 
together they define a morphism $x_1, \ldots, x_n \colon T \to X_1 \times \ldots \times X_n$.
Then we write an expression in terms of $x_i$'s and the morphisms between $X_i$'s and $Y_j$'s, which, 
using the conventions above, defines a morphism $T \to Y_1 \times \ldots \times Y_m$.
So, we have defined a map $\Hom (T, X_1 \times \ldots \times X_n) \to \Hom (T, Y_1 \times \ldots \times Y_m)$, 
and one can easily check (for example, by induction on the length of the formula) 
that any such map constructed by the above rules is functorial in $T$.
So, by Yoneda lemma, we have a morphism $X_1 \times \ldots \times X_n \to Y_1 \times \ldots \times Y_m$.
More explicitly, we can then set $T = X_1 \times \ldots \times X_n$, and set $x_i$ to be the projection onto the $i$th factor
(this will make $x_1, \ldots, x_n = \id_{X_1 \times \ldots \times X_n}$).
Then the expression we have for the morphism $T \to Y_1 \times \ldots \times Y_m$
will already become a morphism $X_1 \times \ldots \times X_n \to Y_1 \times \ldots \times Y_m$, 
which we wanted to define. Then we will not even need the functoriality of the map 
$\Hom (T, X_1 \times \ldots \times X_n) \to \Hom (T, Y_1 \times \ldots \times Y_m)$ in $T$,
and we will not need to actually use Yoneda lemma (but the idea is still based on it, of course).

This way we can write morphisms between schemes in the same way as over an algebraically closed field, 
where it is enough to define the images of rational points.

Now let us do the omitted calculations for Section \ref{sectionpicard} using this approach.
As the first application, 
let us check that the map $\triv_e$ from Section \ref{sectionpicard} is an equivariant isomorphism.
Recall that by the definition of a torsor, 
if $G$ is an algebraic group and $(E, \varphi)$ is a $G$-torsor, then 
$(\varphi, \pr_2) \colon G \times E \to E \times E$
is an isomorphism.
\begin{lemma}\label{triveiso}
If $G$ is an algebraic group, $(E, \varphi)$ is a $G$-torsor, and $e \in E$ is a rational point, 
then $\triv_e=\varphi|_{G \times \{e\}} \colon G \to E$ is a $G$-equivariant isomorphism for the left action of $G$ on itself.
We have $\triv_e^{-1}=\pr_1 \circ (\varphi, \pr_2)^{-1}|_{E \times \{e\}}$.
\end{lemma}
\begin{proof}
To check that $\triv_e$ is an isomorphism, we have to check that for our formulas of $\triv_e$ and $\triv_e^{-1}$, we have 
$\triv_e \circ \triv_e^{-1} = \id_E$ and $\triv_e^{-1} \circ \triv_e = \id_G$.

For the first composition, we have to check the equality of two maps from $E$ to $E$. 
So, for our approach to notation, let us use $E$ as the test scheme, 
and denote the identity map on $E$ (``the projection of the product of a single factor $E$ to the first factor'') by $x$.
Very informally, this $x$ plays the role of an arbitrary rational point of $E$ ``if the base field had been algebraically closed''.
First, we have
$\triv_e (\triv_e^{-1}(x))=\varphi ( \pr_1 ( (\varphi, \pr_2)^{-1} (x, e) ) , e)$.
Next, since $(\varphi, \pr_2)^{-1}$ is the inverse of $(\varphi, \pr_2)$, we have 
$(\varphi, \pr_2) ((\varphi, \pr_2)^{-1} (x,e)) = (x,e)$, in particular, 
$\pr_2 ((\varphi, \pr_2)^{-1} (x,e)) = e$.
So, we can further write 
\begin{equation*}
\varphi ( \pr_1 ( (\varphi, \pr_2)^{-1} (x, e) ) , e) = 
\varphi ( \pr_1 ( (\varphi, \pr_2)^{-1} (x, e) ) , \pr_2 ( (\varphi, \pr_2)^{-1} (x, e) )) = 
\varphi ( (\varphi, \pr_2)^{-1} (x, e) ) = x,
\end{equation*}
and the first composition is indeed the identity map.

For the second composition, we check the equality of maps from $G$ to $G$, use $G$ as the test scheme, 
and denote the identity map on $G$ by $g$.
We have
\begin{equation*}
\triv_e^{-1} (\triv_e(g))=
\pr_1 ( (\varphi, \pr_2)^{-1} (\varphi(g,e), e) )=
\pr_1 ( (\varphi, \pr_2)^{-1} (\varphi(g,e), \pr_2(g,e)) )=
\pr_1(g,e)=g.
\end{equation*}

Finally, the $G$-equivariance of $\triv_e$ is an equality of two maps from $G \times G$ to $E$: informally, 
we either first act by $G$ on $G$ and then apply $\triv_e$, 
or we first apply $\triv_e$ to the second factor $G$ and then act by the first factor $G$ on the result.
So, let us use $G \times G$ as the test scheme, and denote the projections to the two factors by $g_1$ and $g_2$.
Then we have to check that $\triv_e(g_1g_2) = \varphi(g_1,\triv_e(g_2))$. 
Indeed, since $\varphi$ is an action, we have:
\begin{equation*}
\triv_e(g_1g_2)=\varphi(g_1g_2, e)=\varphi(g_1,\varphi(g_2, e))=\varphi(g_1,\triv_e(g_2)).
\end{equation*}
\end{proof}
With this lemma and with our approach to notation, we can write the following identities for the maps $\triv$, $\trivinv$, and $\trivinvinv$
(recall that we omit the subscript $e$ when it is clear which rational point of $E$ we assume).
\begin{corollary}\label{triveformulas}
For $G$ and $(E,\varphi)$ as in Lemma \ref{triveiso}, for arbitrary maps from a test scheme $g_1, g_2 \colon T \to G$ and $x \colon T \to E$, 
and using multiplicative notation for the action $\varphi$, we have:
\begin{gather*}
\triv(g_1)=g_1e, \quad \trivinv(x) = \triv^{-1}(x), \quad \trivinvinv(x) = (\trivinv(x))^{-1}=(\triv^{-1}(x))^{-1}; \\
g_1 \triv(g_2) = \triv(g_1g_2), \quad g_1 \trivinv (x) = \trivinv (g_1 x), \quad \trivinvinv (x) g_1^{-1} = \trivinvinv (g_1 x).
\end{gather*}
\qed
\end{corollary}

Now let us do the rest of calculations omitted in Section \ref{sectionpicard}.

\begin{proof}[Detailed proof of Lemma \ref{gblisquot}]
Like in the previous proof, now the $G_L$-invariance of $p$ is an equality of two 
maps from $G_L \times E_L \times (G/B)_L$ to $(G/B)_L$.
So, for our approach to notation, let us use $G_L \times E_L \times (G/B)_L$ as the test scheme, 
and denote the projections to the factors by $g$, $x$, and $y$, respectively.
We have to check that $p(\varphi_L(g,x), \xi(g,y))=p(x,y)$ 
(recall that $\xi \colon G_L \times (G/B)_L \to (G/B)_L$
is the action morphism).
In terms of our current notation, we can write $p$ as $p(x,y)=\xi(\trivinvinv_e(x),y)$.
We have (using Corollary \ref{triveformulas}, using multiplicative notation for the actions $\xi$ and $\varphi_L$, and 
assuming the subscript $e$ for $\trivinvinv$):
\begin{equation*}
p(\varphi_L(g,x), \xi(g,y))=
\trivinvinv(gx) (g y)=
(\trivinvinv(x) g^{-1}g) y=
\trivinvinv(x) y=
p(x,y).
\end{equation*}

For the universal property, suppose we have a $G_L$-invariant map $q \colon E_L \times (G/B)_L \to Z$.
We have to prove that there exists a unique map $r \colon (G/B)_L \to Z$ such that $q = r \circ p$.
Note first that the restriction of $p$ to $\{e\} \times (G/B)_L \cong (G/B)_L$ is the identity map on $(G/B)_L$.
So, if $r$ exists, it must equal $q|_{\{e\} \times (G/B)_L}$.

Let us check that we indeed have $q = r \circ p$ for this $r$. We have to check an equality of two maps from $E_L \times (G/B)_L$, 
so we use $E_L \times (G/B)_L$ as the test scheme and denote the projections to the two factors by $x$ and $y$, respectively. 
Using the $G_L$-invariance of $q$, we have (with all the same conventions for calculations)
\begin{equation*}
r(p(x,y))=
q(e,\trivinvinv(x)y)=
q(\trivinv (x)e,y)=
q(\triv(\trivinv(x)),y)=
q(x,y).
\end{equation*}
\end{proof}

\begin{lemma}\label{psitwoissemirep}
In the proof of Lemma \ref{actonebisrepres}, 
the formulas (\ref{newgrep}) together (for all $\sigma \in \Gamma$) indeed form a $\Gamma$-semirepresentation on $L \otimes_K V_\lambda$.
\end{lemma}
\begin{proof}
We have already seen in the proof of Lemma \ref{actonebisrepres} that 
each $\widetilde{\psi}_{2, \sigma}$ is a $\sigma$-semilinear map of vector spaces, 
so all we have to check now is that,
for each $\sigma, \tau \in \Gamma$, 
we have
$\widetilde{\psi}_{2,\tau} \circ \widetilde{\psi}_{2,\sigma}=
\widetilde{\psi}_{2,\tau \sigma}$.
Both sides are $(\tau \sigma)$-semilinear maps from the vector space $L \otimes_K V_\lambda$ to itself.
This time, in the calculations, we are not going to talk about morphisms between schemes. 
We are going to talk about (maybe, arbitrary) vectors from the vector space $L \otimes_K V_\lambda$ and 
(fixed) rational points of $G_L$ and of $E_L$.
So, we don't use Yoneda's lemma-based notation this time, and fix an arbitrary vector 
$x \in L \otimes_K V_\lambda$.
We use multiplicative notation for the representation of $G_L$ in $L \otimes_K V_\lambda$
and omit the subscript $e$ for $\trivinvinv$. We have:
\begin{multline*}
\widetilde{\psi}_{2,\tau} \big(\widetilde{\psi}_{2,\sigma}(x)\big) =
\Xi_{\trivinvinv_e(\psi_2 (\tau, e))} \bigg(\widetilde{\psi}_{1,\tau} 
  \Big(\Xi_{\trivinvinv_e(\psi_2 (\sigma,e))} \big(\widetilde{\psi}_{1,\sigma}(x)\big)\Big)\bigg) = 
\trivinvinv\big(\psi_{2,\tau} (e)\big) \widetilde{\psi}_{1,\tau} \Big(\trivinvinv\big(\psi_{2,\sigma}(e)\big) \widetilde{\psi}_{1,\sigma}(x)\Big) = \\
\trivinvinv\big(\psi_{2,\tau} (e)\big) \psi_{1,\tau} \Big(\trivinvinv\big(\psi_{2,\sigma} (e)\big)\Big) 
  \widetilde{\psi}_{1,\tau} \big(\widetilde{\psi}_{1,\sigma}(x)\big)= 
\underbrace{\bigg(\trivinvinv\big(\psi_{2,\tau} (e)\big) \psi_{1,\tau} \Big(\trivinvinv\big(\psi_{2,\sigma} (e)\big)\Big)\bigg)}_{\text{first factor}} 
  \widetilde{\psi}_{1,\tau\sigma}(x).
\end{multline*}
Now let us do some calculations with the first 
factor
in the last formula, using multiplicative notation for $\varphi_L$.
\begin{multline*}
\trivinvinv\big(\psi_{2,\tau} (e)\big) \psi_{1,\tau} \Big(\trivinvinv\big(\psi_{2,\sigma} (e)\big)\Big)=
\Big(\trivinv\big(\psi_{2,\tau} (e)\big)\Big)^{-1} \psi_{1,\tau} \bigg(\Big(\trivinv\big(\psi_2(\sigma, e)\big)\Big)^{-1}\bigg)=\\
\Big(\trivinv\big(\psi_{2,\tau} (e)\big)\Big)^{-1} \bigg(\psi_{1,\tau} \Big(\trivinv\big(\psi_2(\sigma, e)\big)\Big)\bigg)^{-1}=
\bigg(\psi_{1,\tau} \Big(\trivinv \big(\psi_{2,\sigma} (e)\big)\Big) \trivinv\big(\psi_{2,\tau} (e)\big)\bigg)^{-1}=\\
\bigg(\trivinv\Big(\psi_{1,\tau} \big(\trivinv(\psi_{2,\sigma} (e))\big) \psi_{2,\tau} (e)\Big)\bigg)^{-1}=
\trivinvinv \bigg(\psi_{2,\tau} \Big(\trivinv\big(\psi_{2,\sigma} (e)\big)e\Big)\bigg)=\\
\trivinvinv \bigg(\psi_{2,\tau} \Big(\triv\big(\trivinv(\psi_{2,\sigma} (e))\big)\Big)\bigg)=
\trivinvinv \Big(\psi_{2,\tau} \big(\psi_{2,\sigma} (e)\big)\Big)=
\trivinvinv \big(\psi_{2,\tau \sigma} (e)\big).
\end{multline*}
So, for the whole formula for $\overline{\psi}_{2,\tau} (\overline{\psi}_{2,\sigma} (x))$, we get
\begin{equation*}
\widetilde{\psi}_2 \big(\tau, \widetilde{\psi}_2 (\sigma, x)\big) =
\trivinvinv\big(\psi_{2,\tau \sigma} (e)\big) \widetilde{\psi}_{1,\tau\sigma} (x)=
\widetilde{\psi}_{2, \tau\sigma} (x).
\end{equation*}
\end{proof}

\begin{proof}[Detailed proof of Lemma \ref{pequivar}]
In this proof, we also use the notation based on Yoneda lemma for semilinear morphisms. 
All products are still considered in $\Sch_L$. 
If $x$ and $y$ are two $\sigma$-semilinear 
morphisms defined on the same scheme and \emph{with the same $\sigma$}, 
then the notation $x,y$ described in the beginning of this section still makes sense due to Lemma \ref{propdirectproducts}
and denotes a $\sigma$-semilinear morphism.

For the equivariance of $p$, fix an arbitrary $\sigma \in \Gamma$.
We have to check that 
\begin{equation*}
\overline{\psi}_{2,\sigma} \circ p = 
p \circ (\psi_{2,\sigma} \times \overline{\psi}_{1,\sigma}).
\end{equation*}
These are two $\sigma$-semilinear morphisms $E_L \times (G/B)_L \to (G/B)_L$, so 
we use $E_L \times (G/B)_L$ as the test scheme and denote the projections to the factors by $x$ and $y$.
We have (like usually, using multiplicative notation for $\xi$ and omitting the subscript $e$)
\begin{multline*}
\overline{\psi}_{2,\sigma}\big(p(x,y)\big)=
\xi \bigg(\trivinvinv_e\big(\psi_2(\sigma, e)\big), \overline{\psi}_{1,\sigma} \Big(\xi\big(\trivinvinv_e(x),y\big)\Big)\bigg)= 
\trivinvinv\big(\psi_{2,\sigma} (e)\big) \overline{\psi}_{1,\sigma} \big(\trivinvinv(x)y\big)= \\
\underbrace{\trivinvinv\big(\psi_{2,\sigma} (e)\big) \psi_{1,\sigma}\big(\trivinvinv(x)\big)}_{\text{first factor}}\overline{\psi}_{1,\sigma}(y).
\end{multline*}
Like in the previous proof, let us do some calculations with the first 
factor
in the last formula, using multiplicative notation for $\varphi_L$.
\begin{multline*}
\trivinvinv\big(\psi_{2,\sigma} (e)\big) \psi_{1,\sigma}\big(\trivinvinv(x)\big)=
\Big(\trivinv\big(\psi_{2,\sigma} (e)\big)\Big)^{-1} \psi_{1,\sigma}\Big(\big(\trivinv(x)\big)^{-1}\Big)=
\Big( \psi_{1,\sigma}\big(\trivinv(x)\big) \trivinv\big(\psi_{2,\sigma} (e)\big)\Big)^{-1} =\\
\bigg( \trivinv\Big(\psi_{1,\sigma}\big(\trivinv(x)\big)\psi_{2,\sigma} (e)\Big)\bigg)^{-1}=
\trivinvinv\Big(\psi_{2,\sigma}\big(\trivinv(x)e\big)\Big) =
\trivinvinv\bigg(\psi_{2,\sigma}\Big(\triv\big(\trivinv(x)\big)\Big)\bigg) =
\trivinvinv\big(\psi_{2,\sigma}(x)\big).
\end{multline*}
Now, for the whole formula for $\overline{\psi}_{2,\sigma}(p(x,y))$, we have 
\begin{equation*}
\overline{\psi}_{2,\sigma}\big(p(x,y)\big)=
\trivinvinv\big(\psi_{2,\sigma}(x)\big) \overline{\psi}_{1,\sigma}(y)=
p\big(\psi_{2,\sigma}(x), \overline{\psi}_{1,\sigma}(y)\big).
\end{equation*}

For the second statement, we also fix an arbitrary $\sigma \in \Gamma$. 
The $\Gamma$-equivariance of $p$ (which we have just proved) means that 
\begin{equation*}
p = 
\overline{\psi}_{2,\sigma^{-1}} \circ p \circ
(\psi_{2,\sigma} \times \overline{\psi}_{1,\sigma}).
\end{equation*}
The $\Gamma$-equivariance of $q$ means that 
\begin{equation*}
q = 
\psi_{3,\sigma^{-1}} \circ q \circ
(\psi_{2,\sigma} \times \overline{\psi}_{1,\sigma}).
\end{equation*}
Therefore, in addition to $q = r \circ p$, for 
\begin{equation*}
r' := 
\psi_{3,\sigma^{-1}} \circ r \circ
\overline{\psi}_{2,\sigma},
\end{equation*}
we also have $q = r' \circ p$. By the uniqueness in the universal property, we have $r'=r$.
Repeating this argument for all $\sigma \in \Gamma$, we get the $\Gamma$-equivariance of $r$.
\end{proof}

\section{Idea of proof of Proposition \ref{propositionpicards} using \'etale sheaves}\label{appendixetale}

Roughly speaking, the idea of this section is to define several \'etale sheaves on 
the quotient of a torsor modulo the Borel subgroup,
to prove that they form an exact sequence, 
to express 
the Picard group
using the long exact sequence of cohomology, and, finally, to see how this construction behaves when we extend the base field.
For a reference on \'etale sheaves, see \cite[Chapters I and II]{sga45}.


Let us recall some basic properties of restriction of scalars and of \'etale morphisms.
Recall that if $L/K$ is a finite extension of fields, and $X$ is a scheme over $L$, then 
we denote by $X|_K$ the scheme $X$ with base field restricted to $K$.

\begin{proposition}\label{basicetaledowngrade}
Let $L/K$ be a finite extension of fields. Then:
\begin{enumerate}
\item If the extension $L/K$ is separable, then $L$ is an \'etale $K$-algebra. This also means that the dual morphism $(\Spec L)|_K \to \Spec K$
is \'etale.
\item 
\label{canonicaldowngrade}
If $X$ is a variety over $K$, then there is a canonical surjective morphism $(X_L)|_K \to X$, 
dual to $A \to A \otimes L$, $a \mapsto a \otimes 1$ for a $K$-algebra $A$.
\item 
\label{canonicaldowngradeprop}
If $X \to Y$ is an \'etale morphism of schemes over $K$, and $Z \to Y$ is any morphism of schemes over $K$, 
then the projection $X \times_Y Z \to Z$ is \'etale.
If the extension 
$L/K$ 
is separable, then the morphism $(X_L)|_K \to X$ defined above is \'etale.
\item 
\label{etaleopenandcomposition}
An open immersion is an \'etale map. The composition of two \'etale maps is \'etale.
\item 
\label{etalecoverintegral}
If $X \to Y$ is an \'etale morphism, $X$ is connected, and $Y$ is normal and integral, then $X$ is also normal and integral 
(\cite[Proposition I.9.2 and Corollary I.9.11]{sga1}).\qed
\end{enumerate}
\end{proposition}


\begin{remark}\label{downgradeonefactor}
\begin{enumerate}
\item \label{downgradeonefactorformula}
Let $L/K$ be a finite extension of fields. For any $K$-algebras $A$ and $C$, 
$L$-algebra $D$, and morphisms (over $K$) $A \to C$ and (over $L$) $A\otimes L \to D$, 
we have an isomorphism $(C \otimes L) \otimes_{A \otimes L} D \to C \otimes_A D$, 
$(c \otimes l) \otimes d \mapsto c \otimes ld$.
\item \label{downgradeonefactorfibered} 
Using such isomorphisms for affine charts, we get, for any $K$-varieties $X$ and $Y$, for any $L$-variety $Z$, and for any morphisms $Y \to X$ and $Z \to X_L$,
an isomorphism
$(Y_L \times_{X_L} Z)|_K \cong Y \times_X (Z|_K)$.
(For the latter fibered product, we first apply restriction of scalars to the morphism $Z \to X_L$, and then compose the morphism $Z|_K \to (X_L)|_K$
we get with the standard morphism $(X_L)|_K \to X$ from Proposition \ref{basicetaledowngrade} (\ref{canonicaldowngrade}), 
thus getting $Z|_K\to X$.) This isomorphism is functorial in $Y$ and in $Z$
(and even in $X$, although we don't need this).
\item \label{downgradeonefactordirect}
Taking $A=K$ at the very first step, we get an (again functorial) isomorphism of direct products rather than fibered products:
$(Y_L \times Z)|_K \cong Y \times (Z|_K)$.
\end{enumerate}
\end{remark}




Before we construct \'etale sheaves on $E/B$, let us introduce notation for a few maps related to $E/B$.
First, recall that $E$ was originally defined over a field we 
denoted by $K$, 
and that $E/B$ was originally defined as the categorical quotient $(E \times G/B) / G$, 
which means that we have a ($G$-invariant) categorical quotient map $E \times G/B \to E/B$. 
Let $\ell \in G/B$ again denote the point with stabilizer $B$.
We can compose the ($B$-equivariant) embedding $E \cong E \times \{\ell\} \hookrightarrow E \times G/B$ 
with the categorical quotient map and get a $B$-invariant map $E \to E/B$. Denote this map by $\pi$.
Second, denote the restriction of the action map $\varphi \colon G \times E \to E$ to $B \times E$ by $\varphi_B$.


Now, recall $L$ was a finite Galois extension of $K$ such that $E_L$ is a trivial torsor, 
and that $E/B$ was constructed as the Galois descent of $(G/B)_L$ for a certain semiaction (Proposition \ref{ebconstruction}).
Therefore, we have a canonical surjective map $((G/B)_L)|_K \to E/B$ as explained in Proposition \ref{basicetaledowngrade} (\ref{canonicaldowngrade}).
Denote it by $\rho$.
It also follows from the explicit construction of $E/B$ 
(now we will really need the explicit construction, the definition of $E/B$ as a categorical quotient alone is not enough)
that the map 
$\pi_L \circ (\trivinvinv_e)^{-1} \colon G_L \to (G/B)_L$ 
coincides with the standard orbit map $G_L \to (G/B)_L$.
Denote this orbit map by $\tilde{\pi}$. Finally, denote the right action map of $B_L$ on $G_L$ by 
$\tilde{\varphi}_B$ 
(i.e., $\tilde{\varphi}_B \colon B_L \times G_L \to G_L$, $\tilde{\varphi}_B (b,g) = gb^{-1}$ for rational points).

With this notation, let us find a good open covering for $(G/B)_L$.
\begin{lemma}\label{bruhatcovering}
There exists an finite affine open covering $\{\tilde{U}_i\}$ of $(G/B)_L$ such that for each $i$, $\tilde{\pi}^{-1}(\tilde{U}_i)$ 
can be $B_L$-equivariantly identified with $B_L \times \tilde{U}_i$. The $B_L$-equivariance here is for 
$\tilde{\varphi}_B$ 
on $\tilde{\pi}^{-1}(\tilde{U}_i)$ 
and for the action by left multiplication on the first factor on $B_L \times \tilde{U}_i$. 
After this identification, $\tilde{\pi}$ becomes the projection to the second factor.
\end{lemma}
\begin{proof}
Follows from Bruhat decomposition. 
In more details, let $U^-$ be a maximal unipotent subgroup of $G_L$ opposite to $B_L$.
By \cite[Theorem 21.84]{milne}, the multiplication map $U^- \times B_L \to G_L$ is an open immersion of the big Bruhat cell, 
so we can use $\tilde{U}_1:=\tilde{\pi}(U^- B_L) \cong U^-$ as one of the sets in the open covering.
(To be precise, we have to apply inversion to $B_L$ in $U^- \times B_L$
to get the desired $B_L$-equivariance.) The rest of open covering can be constructed by shifts by the Weyl group.
%
\end{proof}
Now let us find a good \'etale covering for $E/B$.
\begin{lemma}\label{directproductrestriction}
Suppose that $\tilde{U} \subset (G/B)_L$ is an open subset such that $\tilde{\pi}^{-1}(\tilde{U})$ can be 
$B_L$-equivariantly (like in the previous lemma) identified with $B_L \times \tilde{U}$. Let $U=\tilde{U}|_K$.
Then $\rho|_U$ is an \'etale map, and $E \times_{E/B} U$ (for the maps $\pi$ and $\rho|_U$) 
can be $B$-equivariantly identified with $B \times U$.
The $B$-equivariance here is for the actions of $B$ on the first factors, the action of $B$ on itself is by left multiplication.
The projections to the second factor ($U$) commute with this identification.
\end{lemma}
\begin{proof}
First, the \'etality follows from the basic properties of \'etale maps 
(Proposition \ref{basicetaledowngrade} (\ref{canonicaldowngradeprop}, \ref{etaleopenandcomposition})).
Next, by Remark \ref{downgradeonefactor} (\ref{downgradeonefactorfibered}, \ref{downgradeonefactordirect}), we have
\begin{equation*}
E \times_{E/B} U \cong (E_L \times_{(E/B)_L} \tilde{U})|_K \cong (G_L \times_{(G/B)_L} \tilde{U})|_K \cong (\tilde{\pi}^{-1}(\tilde{U}))|_K \cong 
(B_L \times \tilde{U})|_K \cong B \times U.
\end{equation*}
(For the second isomorphism, we use 
$(\trivinvinv_e)^{-1}$
for the isomorphism between $E_L$ and $G_L$.)
The $B$-equivariance is a direct computation using functoriality in 
Remark \ref{downgradeonefactor} (\ref{downgradeonefactorfibered}, \ref{downgradeonefactordirect})
and the explicit formula in Remark \ref{downgradeonefactor} (\ref{downgradeonefactorformula}).
\end{proof}
From now on, for the covering $\{\tilde{U}_i\}$ from Lemma \ref{bruhatcovering}, 
denote $U_i=\tilde{U_i}|_K$. With the maps $\rho|_{U_i}$, this is an \'etale covering of $E/B$ (recall that $\rho$ is surjective).
We also get an \'etale covering $\{E \times_{E/B} U_i\}$ of $E$.

The next step will be to introduce several \'etale sheaves on $E/B$ and maps between them. To do that, we need three more preliminary lemmas.
\begin{lemma}
The map $\pi \colon E \to E/B$ is flat.
\end{lemma}
\begin{proof}
Let us prove that $\pi$ is flat at each point of $E$. Let $x \in E$. 
Choose $U_i$ so that $x$ is in the image of the map $E \times_{E/B} U_i \to E$ (denote this map temporarily by $f$), 
$x=f(y)$, $y \in E \times_{E/B} U_i$. 
The composition $\pi \circ f$ is flat, because it can be written as 
$E \times_{E/B} U_i \cong B \times U_i \to U_i \to E/B$, a composition of the projection to a direct factor and an \'etale map.
In particular, $\pi \circ f$ is flat at $y$.
Also, $f$ is \'etale, so also flat. By \cite[Corollary 2.2.11.iv]{ega4}, $\pi$ is flat at $x$.
\end{proof}
\begin{lemma}\label{aboveconnected}
Let $f \colon V \to E/B$ be an \'etale map, where $V$ is connected. Then $E \times_{E/B} V$ is also connected.
\end{lemma}
\begin{proof}
If $f$ factors through one of the $U_i$s (i.e., if there exists an \'etale map $g \colon V \to U_i$ such that $f = \rho|_{U_i} \circ g$), 
then by Lemma \ref{directproductrestriction}, we can write 
$E \times_{E/B} V \cong E \times_{E/B} U_i \times_{U_i} \times V \cong B \times U_i \times_{U_i} \times V 
\cong B \times V$. The latter is connected since 
by \cite[Theorem 21.11 (c)]{milne} and \cite[Theorem 21.68 (a,b)]{milne}, $B$ as an abstract variety 
is the product of several copies of $K$ (the affine line) and of $K^*$.


For a general $V$, first recall that $E/B$ is integral, and also smooth hence normal. 
So, by Proposition \ref{basicetaledowngrade} (\ref{etalecoverintegral}), $V$ is also integral hence irreducible.
Next, choose an arbitrary $U_i$, and let $V'$ be a connected component of $V \times_{E/B} U_i$ 
(\'etale maps are flat hence dominant, so this fibered product is nonempty).
Assume $E \times_{E/B} V$ is not connected and is a disjoint union of $W_1$ and $W_2$. 
The projection $E \times_{E/B} V \to V$ is flat since $\pi$ is flat by the previous lemma. 
So, each $W_i$ projects dominantly to $V$ by \cite[Lemma 4.3.7]{qingliu} (here we use the irreducibility of $V$).
Similarly, $V'$ projects to $V$ dominantly, and the fibered products $W_i \times_V V'$ are nonempty.

Then we have $E \times_{E/B} V' \cong E \times_{E/B} V \times_V V' \cong (W_1 \times_V V') \sqcup (W_2 \times_V V')$.
On the other hand, the map $V' \to V \to E/B$ can be rewritten as $V' \to U_i \to E/B$, so $E \times_{E/B} V'$ 
is connected by the first step, a contradiction.
\end{proof}
Denote the character group of $B$ by $\X(B)$. 
\begin{lemma}\label{productinvertiblesplitting}
If $W$ is a connected and integral variety over $K$, then any invertible function on $B \times W$ 
can be uniquely written as $\lambda f$, where $\lambda \in \X(B)$, 
and $f$ is an invertible function on $W$.
\end{lemma}
\begin{proof}
Follows from the structure of $B$ as of an abstract variety again.
\end{proof}
(The statement for direct products with $K$ or $K^*$ is quite straightforward, but there also not so obvious related known results for direct 
products under different conditions, 
cf.\ \cite[Lemma 10]{colliotthelenesansuc}, \cite[Lemma 2.1]{fossumiversen}, or \cite[Theorem 2]{rosenlicht}.)

Now, let us construct a sequence of four \'etale sheaves on $E/B$ and three maps between them.
For the last two maps, we will follow the ideas of \cite[Lemma 6.4]{sansuc},
but for the convenience of readers, we present the whole construction here.

First, given an \'etale map $V \to E/B$ and an invertible function $f$ on $V$, 
consider the projection $E \times_{E/B} V \to V$ (here, for the fibered product, we use $\pi$ for the map $E \to E/B$).
We can compose this projection with $f$ and get an invertible function on $E \times_{E/B} V$.
This defines a map of \'etale sheaves $\mathbf G_{m, E/B} \to \pi_* \mathbf G_{m,E}$.
Denote this map by $s_1$.

Second, using 
$\varphi_B \colon B \times E \to E$ 
instead of $\pi$, we similarly get a 
map of \'etale sheaves 
$\mathbf G_{m, E} \to \varphi_{B*} \mathbf G_{m,B \times E}$ on $E$. 
Denote by $s_2$ 
the $\pi$-pushforward of this map. 

Third, we are going to construct a map from 
$\pi_* \varphi_{B*} \mathbf G_{m, B \times E}$
to the 
constant sheaf 
$\uline{\X(B)}$.
If we denote by $\pr_2$ the projection to the second factor $B \times E \to E$, then 
$\pi \circ \pr_2 = \pi \circ \varphi_B$
by $B$-invariance of $\pi$.
So, for any \'etale map $V \to E/B$, we have a canonical isomorphism $(B \times E) \times_{E/B} V \cong B \times (E \times_{E/B} V)$
(where, in the first fibered product, we use 
$\pi \circ \pr_2 = \pi \circ \varphi_B$ 
for the map $B \times E \to E/B$).

If $V$ is connected, then $E \times_{E/B} V$ is also connected 
by Lemma \ref{aboveconnected}
and integral by Proposition \ref{basicetaledowngrade} (\ref{etalecoverintegral})
applied to the \'etale map $E \times_{E/B} V \to E$.
By Lemma \ref{productinvertiblesplitting}, 
any invertible function on $B \times (E \times_{E/B} V)$ can be uniquely written as $\lambda f$, where $\lambda \in \X(B)$, 
and $f$ is an invertible function on $E \times_{E/B} V$
So, we can define a map 
$s_3 \colon \pi_* \varphi_{B*} \mathbf G_{m, B \times E} \to \uline{\X(B)}$
by saying $s_3(\lambda f)=\lambda$.

With these maps of \'etale sheaves, let us prove two lemmas.
These two lemmas are similar to \cite[Proposition 6.10]{sansuc}, but are proved under slightly different conditions. 
While the variety with a group action here is $E$, which is, unlike in \cite[Proposition 6.10]{sansuc}, not at all arbitrary, 
the acting group here is $B$, which is not reductive 
(and in \cite[Proposition 6.10]{sansuc}, the acting group has to be reductive if we want to keep the base field arbitrary).
Also, \cite[Proposition 6.10]{sansuc} would require that $E$ be a $B$-torsor over $E/B$, while here we don't prove this 
(although in fact this is true) and don't use the definition of a torsor over anything other than a point at all.
Finally, we will still need an explicit description of an isomorphism (which we will establish in Lemma \ref{picardetale}) 
between $\X(B)$ and $\Pic(E/B)$ to check that it commutes with base field change, 
so it would not be possible to use \cite[Proposition 6.10]{sansuc} ``as a black box'' even if it allowed non-reductive groups.

\begin{lemma}
The sequence of \'etale sheaves on $E/B$
\begin{equation}\label{etaleshort}
0 \to \mathbf G_{m} \xrightarrow{s_1} \pi_* \mathbf G_{m,E} \xrightarrow{s_3 \circ s_2} \uline{\X(B)} \to 0
\end{equation}
is exact.
\end{lemma}
\begin{proof}
We check the exactness on the \'etale covering sieve generated by $\{U_i\}$. Let $f \colon V \to U_i$ be an \'etale map.
We have to verify the exactness of the sequence of sections on $V$, and it is enough to consider the case when $V$ is connected.
By Lemma \ref{directproductrestriction} (and Lemma \ref{bruhatcovering}), we have:
$E \times_{E/B} V \cong E \times_{E/B} U_i \times_{U_i} V \cong B \times U_i \times_{U_i} V \cong B \times V$.
These isomorphisms are $B$-equivariant and commute with the projection to $V$.
The rest is a direct computation (using Lemma \ref{productinvertiblesplitting} again, for $W= V$ this time).
\end{proof}

\begin{lemma}\label{picardetale}
The sequence (\ref{etaleshort}) induces an isomorphism between $\mathfrak X(B)$ and $\Pic(E/B)$.
\end{lemma}
\begin{proof}
We have the following piece of the long exact sequence of cohomology for (\ref{etaleshort}):
\begin{equation}\label{etalelong}
H^0(E/B,\mathbf G_{m}) \to H^0(E/B,\pi_* \mathbf G_{m,E}) \to H^0(E/B,\uline{\X(B)}) \to H^1(E/B,\mathbf G_{m}) \to H^1(E/B,\pi^* \mathbf G_{m,E}).
\end{equation}
In the second term, $H^0(E/B,\pi_* \mathbf G_{m,E}) = \mathcal O(E)^*$, 
all constant functions come from $H^0(E/B,\mathbf G_{m})$. 
And nonconstant invertible functions on $E$ don't exist, 
since any such function would give a nonconstant invertible function 
on $E_L \cong G_L$, a split simple algebraic group, which is not possible by \cite[Corollary 2.2]{fossumiversen}.
(More precisely, \cite[Corollary 2.2]{fossumiversen} requires the base field to be algebraically closed, but this is not a problem since 
a nonconstant invertible function on $G_L$ will stay nonconstant invertible if we extend $L$ to an algebraic closure.)
So, the morphism 
$H^0(E/B,\pi_* \mathbf G_{m,E}) \to H^0(E/B,\uline{\X(B)})$ 
in (\ref{etalelong}) is~0.

The last term, $H^1(E/B,\pi^* \mathbf G_{m,E})$, is the term $E_2^{1,0}$ of 
a Leray spectral sequence converging to $H^\bullet (E, \mathbf G_{m})$ (see \cite[II.3.4]{sga45}).
The differentials around $E_i^{1,0}$ in a spectral sequence are from $E_i^{1-i,i-1}$ and to $E_i^{1+i,-i+1}$, 
so at least one superscript is negative in each of these terms for $i \ge 2$. In this particular spectral sequence, 
we have $E_2^{j,k} = H^j (E/B,R^k \pi_* \mathbf G_{m,E})$, so the terms with at least one negative superscript are all zero.
So, $H^1(E/B,\pi^* \mathbf G_{m,E})$ is actually a subquotient (in fact, a subgroup) of $H^1 (E, \mathbf G_{m})$.

By \cite[Proposition II.2.3]{sga45}, $H^1 (E, \mathbf G_{m}) \cong \Pic(E)$. 
By Lemma \ref{picardexact} (\ref{picardtrueexact}), $\Pic(E)$ can be embedded into $\Pic(E_L) \cong \Pic(G_L)$.
By \cite[Corollary 18.24]{milne} (see also \cite[Proposition 21.64]{milne}), 
the Picard group of a simply connected split simple algebraic group is trivial, so 
$\Pic(G_L) = 0$, and $H^1(E/B,\pi^* \mathbf G_{m,E}) = 0$.
Therefore, in (\ref{etalelong}), we have an isomorphism between 
$H^0(E/B,\uline{\X(B)})=\X(B)$ 
and $H^1(E/B,\mathbf G_{m})$. By \cite[Proposition II.2.3]{sga45} again, 
$H^1(E/B,\mathbf G_{m}) \cong \Pic(E/B)$.
\end{proof}

\begin{proof}[(Last steps of the) Proof of Proposition \ref{propositionpicards}]
It suffices to check that the isomorphism in Lemma \ref{picardetale} commutes with field extensions, and this is done by a direct computation.
In more details, $H^1(E/B,\mathbf G_{m})$ can be interpreted as $\Ext^1_{E/B} (\underline{\ZZ}, \mathbf G_{m})$, 
and then the isomorphism $\Pic(E/B) \to H^1(E/B,\mathbf G_{m})$ can be written explicitly.
Namely, for a line bundle $\mathcal L$ on $E/B$ and for an \'etale map $f \colon V \to E/B$, where $V$ is connected and $f^* \mathcal L$ is trivial, 
we take the following extension of $\ZZ$ by 
$\mathbf G_{m,E/B} (V)=\mathcal O(V)^*$:
\begin{equation*}
\bigsqcup_{n \in \ZZ} \Isom (\mathcal O_V, f^* \mathcal L^{\otimes n}).
\end{equation*}
(If $f^* \mathcal L$ is not trivial, we simply take 
$\mathcal O(V)^*$
instead of the extension and map it to $0 \in \ZZ$: for a short sequence of 
sheaves to be exact, it is enough to have a short exact sequence of groups for each \'etale $f \colon V \to E/B$ from a covering sieve, 
not necessarily for a completely arbitrary \'etale $f \colon V \to E/B$.
In this case, the covering sieve consists of the $f$s such that $f^* \mathcal L$ is trivial.)

Then we have to check the commutativity of the following diagram for any field extension $K'/K$
(here, we also interpret $H^0(E/B,\uline{\X(B)})$ as $\Hom_{E/B}(\underline{\ZZ}, \uline{\X(B)})$):
\begin{equation*}
\xymatrix{
\X(B_{K'}) \ar[r]^-{\sim} \ar@{=}[d] & \Hom_{(E/B)_{K'}}(\underline{\ZZ}, \uline{\X(B_{K'})}) \ar[r]^-{\sim} & 
  \Ext^1_{(E/B)_{K'}} (\underline{\ZZ}, \mathbf G_{m}) & \Pic ((E/B)_{K'}) \ar[l]_-{\sim} \\
\X(B) \ar[r]^-{\sim} & \Hom_{E/B}(\underline{\ZZ}, \uline{\X(B)}) \ar[r]^-{\sim} & 
  \Ext^1_{E/B} (\underline{\ZZ}, \mathbf G_{m}) & \Pic (E/B) \hbox to 0pt{\hspace{1.5em}.\hss} \ar[l]_-{\sim} \ar[u]
}
\end{equation*}
To check the commutativity, we fix a character $\lambda \in \X(B)$ and the line bundle $\mathcal L$ on $E/B$ corresponding to $\lambda$
via the bottom line (in other words, we have an isomorphism between the two extensions $\mathcal F$ and $\mathcal G$ 
of $\underline{\ZZ}$ by 
$\mathbf G_{m}$ on $E/B$
corresponding to $\lambda$ and to $\mathcal L$). 
Then we use this isomorphism to establish an isomorphism between the two extensions $\mathcal F'$ and $\mathcal G'$ of 
$\underline{\ZZ}$ by 
$\mathbf G_{m}$ on $(E/B)_{K'}$
corresponding to $\lambda$ and to $\mathcal L_{K'}$.
Namely, for any \'etale $f \colon V \to E/B$ such that $f^* \mathcal L$ is trivial, we check that 
$\mathcal F'(V_{K'})$ (resp.\ $\mathcal G'(V_{K'})$) can be treated as a pushout 
of 
$\mathbf G_{m,(E/B)_{K'}} (V_{K'}) = \mathcal O(V_{K'})^*$ 
and $\mathcal F(V)$ (resp.\ $\mathcal G(V)$). 
Also, for any \'etale $g \colon W \to V$ over $K$, $\mathcal F(W)$ (resp.\ $\mathcal G(W)$) is a pushout of 
$\mathbf G_{m,E/B} (W)$
and $\mathcal F(V)$ (resp.\ $\mathcal G(V)$). 
A similar equality holds for any \'etale $g' \colon W' \to V_{K'}$ over $K'$ 
(here, $W'$ is a scheme over $K'$, it does not have to be the extension of scalars of anything over $K$).
Then we 
use standard properties of pushout and 
the axiom of an (\'etale) sheaf.
Further details omitted.
\end{proof}

\end{document}